\documentclass{article}
\usepackage{graphicx} 
\usepackage{tikz}

\usepackage{amsmath}
\usepackage{amsfonts}
\usepackage{amssymb}
\usepackage{amsthm}
\usepackage{setspace}
\usepackage{hyperref}
\usepackage{graphicx}
\usepackage{stmaryrd}
\usepackage{mathtools}
\usepackage{subcaption}

\usepackage{stackengine}

\newtheorem{theorem}{Theorem}[section]
\newtheorem{proposition}[theorem]{Proposition}
\newtheorem{lemma}[theorem]{Lemma}
\newtheorem{corollary}[theorem]{Corollary}

\theoremstyle{definition}
\newtheorem{definition}[theorem]{Definition}
\newtheorem{example}[theorem]{Example}
\newtheorem{remark}[theorem]{Remark}
\newtheorem{conjecture}[theorem]{Conjecture}

\newtheorem{notation}[theorem]{Notation}

\newcommand{\Aut}{\mathrm{Aut}}
\newcommand{\Out}{\mathrm{Out}}
\newcommand{\Inn}{\mathrm{Inn}}

\title{Type VF for outer automorphism groups of large-type Artin groups}
\author{Oli Jones \footnote{Email: oj2003@hw.ac.uk. Address: School of Mathematical and Computer Sciences,
Heriot-Watt University, Edinburgh, Scotland, EH14 4AS}.}
\date{October 2024}

\begin{document}

\maketitle

\begin{abstract}
    Given a connected large-type Artin group $A_\Gamma$, we introduce a deformation space $\mathcal{D}$. If $\Gamma$ is triangle-free, or has all labels at least 6, we show that this space is canonical, in that it depends only on the isomorphism type of $A_\Gamma$, and admits an $\Out(A_\Gamma)$-action. Using this action we conclude that $\Out(A_\Gamma)$ is of type VF, which implies $\Out(A_\Gamma)$ finitely presentable. We emphasise that our proof can handle cases where $\Gamma$ has separating vertices, which were previously problematic.

    In fact, our proof works for all connected large-type Artin groups satisfying the technical condition of having rigid chunks. We conjecture that all connected large-type Artin groups have rigid chunks, and therefore outer automorphism groups of type VF.
\end{abstract}

\noindent \rule{7em}{.4pt}\par

\small

\noindent 2020 \textit{Mathematics subject classification.} 20E08, 20F28, 20F36, 20F65.

\noindent \textit{Key words.} Artin groups, outer automorphisms, deformation spaces, actions on trees.

\normalsize

\section{Introduction}

An Artin group is usually defined via a presentation.  Let $\Gamma$ be a simplicial graph with vertex set $V(\Gamma)$ and edge set $E(\Gamma)$, and suppose that each edge $\{a,b\}$ is given a coefficient $m_{ab} \in \{2, 3, 4, \cdots \}$. We call this a defining graph. Then the Artin group associated with $\Gamma$ is the group given by the following presentation:
$$A_{\Gamma} \coloneqq \langle V(\Gamma) \ | \ \underbrace{aba\cdots}_{m_{ab} \text{ terms}} = \underbrace{bab\cdots}_{m_{ab} \text{ terms}}, \forall \{a,b\} \in E(\Gamma) \rangle.$$  If $\Gamma$ is connected, we say the Artin group is \emph{connected}. If each coefficient satisfies $m_{ab} \geq 3$, we say the Artin group is of \emph{large-type}.

It is possible that there are non-isomorphic defining graphs giving rise to isomorphic Artin groups. As such, in this paper we will prefer to think of abstract Artin groups $A$, with no choice of defining graph $\Gamma$ and isomorphism $A \cong A_\Gamma$. The isomorphism problem for large-type Artin groups has been completely solved \cite{vaskou2023isomorphism}. In particular, the properties of being a connected or large-type Artin group are invariant under isomorphism, so can be seen as properties of abstract Artin groups.

In the present work, we show that the outer automorphism groups of many large-type Artin groups have type VF, meaning there is a finite index subgroup with a finite classifying space (see Definition \ref{defF} and Definition \ref{defVF}). This is the strongest possible finiteness condition, in particular implying finite generatability and finite presentability.

\begin{theorem}\label{mainTheorem1}(Theorem \ref{finitePresentability})
    Suppose $A$ is a connected large-type Artin group with rigid chunks. Then $\Out(A)$ is of type VF. 
\end{theorem}

The property of having rigid chunks is a technical condition, and we postpone the precise definition to Definitions \ref{defChunk} and \ref{rigidChunks}. Essentially, this condition means subgraphs of defining graphs for $A$ with no separating edges or vertices (the so called chunks) are isomorphism invariant. We conjecture this technical condition is always satisfied in the connected large-type setting.

\begin{conjecture}\label{rigidChunkConjecture}  
    Connected large-type Artin groups have rigid chunks.
\end{conjecture}

Of course, if this conjecture is true, it immediately follows from our main theorem that $\Out(A)$ is of type VF for any connected large-type Artin group $A$. We remark that Crisp conjectured a generating set for the automorphism group of a connected large-type Artin group \cite{crisp2005automorphisms}, and that if Crisp's conjecture is true, then so is Conjecture \ref{rigidChunkConjecture}.

It is known by work of Crisp \cite{crisp2005automorphisms}, Vaskou \cite{vaskou2023isomorphism}, and Blufstein, Martin, Vaskou \cite{Blufstein2024} that chunk rigidity holds for certain subclasses. We obtain the following corollary, where we say $\Gamma$ is triangle-free if no three edges form a triangle.

\begin{corollary}\label{corollaryxxxlandtree}(Corollary \ref{application})
    Suppose $A_\Gamma$ is a connected large-type Artin group, and $\Gamma$ either is triangle-free, or has every label at least 6. Then $\Out(A_\Gamma)$ is of type VF.
\end{corollary}

In any case of the previous corollary where $\Gamma$ has a separating vertex, finite presentability of $\Out(A_\Gamma)$ is new. If $\Gamma$ is furthermore not triangle-free, even finite generatability is new.

We remark that in general, automorphism groups can be complicated even (or perhaps especially) for seemingly simple groups. For example, the (outer) automorphism group of $\langle x, t \mid tx^2t^{-1} = x^4 \rangle$ is not even finitely generated \cite{Collins1983AutomorphismsAH}.

For a connected large-type Artin group $A$ with rigid chunks, we study $\Out(A)$ via its action on a (finite dimensional deformation retract of a) contractible simplicial complex called a deformation space, a notion originally introduced by Forester \cite{spaces2001forester}. To our knowledge, this is the first time deformation spaces have been considered in the context of outer automorphisms of Artin groups. Before providing further details of our strategy, we will discuss the history of studying outer automorphisms of Artin groups, motivating the adoption of this technology.

\subsection{Historical Context}

Free groups are some of the simplest Artin groups. The study of the automorphism groups of free groups predates even the definition of Artin groups, and continues to be an active area. In \cite{Nielsen1924} it was shown that the outer automorphism groups of free groups are finitely presented. Later, Culler and Vogtmann introduced the so called \emph{Outer Space}, a contractible topological space parameterising free $F_n$-actions on trees, with an $\Out(F_n)$-action \cite{Vogtmann1986}. In the same paper, they used this space (and its cocompact \emph{spine}) to show $\Out(F_n)$ is of type VF. The outer space has been a fundamental tool in the study of $\Out(F_n)$ since its introduction.

Charney, Stambaugh and Vogtmann have now introduced a generalisation of outer space for so called untwisted automorphisms of right-angled Artin groups (RAAGs) \cite{charney17untwisted}. This space is parameterised by actions on complexes generalising the well known Salvetti complexes. Like the classical outer space, the existence of this space, and a cocompact spine, demonstrates that, for each RAAG $A_\Gamma$, the untwisted subgroup $U(A_\Gamma) \leq \Out(A_\Gamma)$ is of type VF. More recently, Bregman, Charney and Vogtmann introduced a full outer space for RAAGs, although there is no cocompact spine. Nonetheless, Day and Wade showed, via different methods, $\Out(A_\Gamma)$ is of type VF for each RAAG $A_\Gamma$ \cite{day19vfraag}.

Beyond RAAGs there has been progress in two main directions. 

One is that of the spherical Artin groups, that is those where the associated Coxeter group quotient is finite, and affine Artin groups. The earliest result in this direction was finding the automorphism groups of braid groups \cite{dyer1981braid}. Around 20 years ago, the outer automorphism groups of spherical dihedral Artin groups \cite{ghrm2000treeaction}, and Artin groups of type $B_n$, $\Tilde{A_{n}}$ and $\Tilde{C_{n}}$ were also calculated \cite{charney2005finiteaffine} (see also \cite{paris2024endomorphismsatilde}\cite{paris2024endomorphismsB}). Recently the (outer) automorphism groups for Artin groups of type $D_n$ were determined \cite{castel2024endomorphismsartingroupstype}. We will not focus on this direction since, in some sense, spherical Artin groups are as far from large as possible, but remark in all cases listed in this paragraph the outer automorphism groups are virtually abelian (often they are finite).

The other direction is the two-dimensional Artin groups, of which the large-type Artin groups are a subclass. Crisp initiated the study of automorphism groups of two-dimensional Artin groups, finding a finite generating set for the isomorphism groupoid of connected large-type triangle-free (that is to say, $A_\Gamma$ of connected large-type where the graph $\Gamma$ is triangle-free) Artin groups \cite{crisp2005automorphisms}. It quickly follows that the automorphism groups are finitely generated. Blufstein, Vaskou and Martin recently proved the same result for XXXL-type Artin groups without separating vertices in $\Gamma$ \cite{Blufstein2024}. Crisp also showed that, in the case that $\Gamma$ (once again connected, large-type and triangle-free) also has no separating vertices, $\Out(\Gamma)$ is virtually abelian. Recently, An and Cho found explicit presentations in this case \cite{an2022automorphismgroupsartingroups}. 

Vaskou calculated the entire automorphism groups for large-type free-of-infinity Artin groups (i.e. those given by clique defining graphs with all labels greater than or equal to 3) \cite{vaskou2023automorphisms}, showing that the outer automorphism groups were finite. Huang, Osajda and Vaskou generalised this to large-type Artin groups admitting a twistless heirarchy terminating in twistless stars \cite{huang2024twistless}.

The Artin groups in the previous paragraph are rigid in the sense that if $A_{\Gamma} \cong A_{\Gamma'}$ where $A_\Gamma$ is large-type and admits a twistless heirarchy terminating in twistless stars, then $\Gamma \cong \Gamma'$. Morally, this is what limits the number of automorphisms and results in $\Out(A_\Gamma)$ being finite. Notably, these defining graphs have no separating vertices or edges. The Artin groups addressed by An and Cho may have separating edges (but not vertices), and it is possible that $A_{\Gamma} \cong A_{\Gamma'}$ where $\Gamma$ and $\Gamma'$ are not isomorphic but ``twisted'' over one of these separating edges. These twists commute, which justifies the outer automorphism groups being virtually abelian. As such, these groups can still be thought of as towards the rigid end of the spectrum of large-type Artin groups.

In the situation when $\Gamma$ has separating vertices, the group structure has remained mysterious. For $\Gamma$ a star ($n$-edges joined along a common vertex) with edges labelled by 3, Crisp showed $\Aut(A_\Gamma)$ contained Braid groups, and that $\Out(A_\Gamma)$ was not virtually abelian \cite{crisp2005automorphisms}. Correspondingly, defining graphs with separating vertices can be very flexible. For instance it follows from Vaskou's solution to the isomorphism problem that if $\Gamma$ is an $n$-edged star with all edges labels 3, $A_\Gamma \cong A_{\Gamma'}$ for any $\Gamma'$ which is a tree with $n$ edges all labeled 3. 

\subsection{Our methods}

In this paper, we develop tools to understand the outer automorphism groups of connected large-type Artin groups $A$ with flexible defining graphs, such as those discussed at the end of the previous section. As advertised, we will introduce a deformation space equipped with an action of $\Out(A)$.  

The notion of a deformation space was first introduced by Forester \cite{spaces2001forester}. Analogously to the successful techniques in the setting of free groups, these are spaces of actions on simplicial trees. For any pair $(T, \Omega)$ of a tree $T$ and an action $\Omega: G \rightarrow \Aut(T)$ of a group $G$, this $G$-tree is contained in a space of certain $G$-trees related by so called elementary deformations. See Definition~\ref{defDeformation}.

The outer space for the free group $F_n$ is an example of a deformation space. In particular, it is the one containing any Cayley graph with respect to any free basis (which of course is an $F_n$-tree). 

Given a $G$-tree $(T, \Omega)$ and an automorphism $\varphi \in \Aut(G)$, one naturally obtains the $G$-tree $(T, \Omega \circ \varphi)$. Deformation spaces have been a powerful tool to study outer automorphism groups when there is a \emph{canonical} deformation space, where every automorphism sends every $G$-tree in the space to another $G$-tree in the space. In this case $\Aut(G)$ acts on the deformation space (and this descends to an action of $\Out(G)$, due to an equivalence relation we impose on $G$-trees). This is the case for the outer space for free groups, but also for the JSJ space \cite{guirardel2016jsj}, and the natural deformation space for generalised Baumslag-Solitar groups \cite{levitt2007automorphism}.

We introduce a deformation space $\mathcal{D}$ for large-type chunk rigid Artin groups. It contains visual splittings of the Artin group, that is splittings as an amalgamated product of parabolic subgroups (see Definition \ref{definitionParabolic}), where the vertex groups are certain parabolic subgroups called ``chunks''. It is precisely the rigidity of these chunks which allows us to prove the following theorem, and allows us to use $\mathcal{D}$ to study the entirety of $\Out(A)$.

\begin{theorem}\label{mainTheorem2}(Theorem \ref{canonical})
    Let $A$ be a large-type Artin group with rigid chunks. Then there is a canonical deformation space $\mathcal{D}$ admitting an $\Out(A)$-action.
\end{theorem}

It follows from general theorems on deformation spaces, of Guirardel and Levitt, that this space is contractible, and deformation retracts onto a finite dimensional simplicial complex $\mathcal{F}$ \cite{deformatio72001guirardel}.

To prove Theorem \ref{mainTheorem1} the first, and most involved, step is to show that the action of $\Out(A)$ on $\mathcal{F}$ is co-compact. Then, we pass to a finite index subgroup $\Out_{0,+}(A) \leq \Out(A)$, which is well defined only because of the rigid chunks. Finally we show that $\Out_{0,+}(A)$ acts with type F simplex stabilisers. Type VF for $\Out(A)$ follows immediately.

The canonical deformation space provided in Theorem \ref{mainTheorem2} is a promising tool for the study of outer automorphism groups of large-type Artin groups beyond studying finiteness properties. For example, the outer space for a free group has the Nielsen realisation property: that is, any finite subgroup of $\Out(F_n)$ fixes a point in the associated outer space \cite{Zimmermann1981freenielsen}. This property is named in analogy with the corresponding fact about mapping class groups acting on Teichm\"{u}ller space \cite{kerckhoff1983nielsen}. The same result holds for finite subgroups of the untwisted subgroup of $U(A) \leq \Out(A)$ where $A$ is a general RAAG \cite{bregman2024raagnielsen}. An immediate consequence of our Theorem \ref{mainTheorem2} is a weak form of Nielsen realisation for large-type Artin groups.

\begin{corollary}
    Suppose $A$ is a connected large-type Artin group with rigid chunks, and $F \leq \Out(A)$ is a finite solvable subgroup. Then $F$ fixes a point on $\mathcal{D}$.
\end{corollary}
\begin{proof}
    This is a direct application of \cite[Corollary 8.4]{deformatio72001guirardel} to our deformation space $\mathcal{D}$.
\end{proof}

A natural question is if a full version of Neilsen realisation (that is, dropping the assumption that $F$ is solvable) holds for a $\Out(A) \curvearrowright \mathcal{D}$.

We emphasise that the separating vertices that were a source of difficulty in previous work are not a problem for our methods. Indeed, one of the subclasses of large-type Artin groups known to have rigid chunks is those $A_\Gamma$ where $\Gamma$ is a tree (since of course such $\Gamma$ are triangle-free), and every internal vertex is separating. Notably, this subclass includes previously discussed examples of Crisp, where $\Gamma$ is a star and $\Out(A_\Gamma)$ is not virtually abelian. 

\section*{Acknowledgements}

The author thanks their supervisor, Laura Ciobanu, for her support and reading of the manuscript. This project would not have been completed without many helpful conversations, thank you to Naomi Andrew, Ruth Charney, Giovanni Sartori and Alexandre Martin. Thank you to Ignat Soroko and Gilbert Levitt for their comments on the first version, and thanks especially to Ric Wade, who suggested the approach to strengthen the main result from finite presentability to type VF.

\section{Preliminaries}

This section is divided into three subesections covering known results we will use. Section \ref{sectionArtin} is about Artin groups, and Section \ref{sectionDeformation} is about deformation spaces, and Section \ref{sectionVF} is about the finiteness property VF.

\subsection{Artin groups}\label{sectionArtin}

Artin groups are conveniently defined via a presentation. For our purposes, it will be important to distinguish between a specific presentation and the underlying group.

\begin{definition}
    A \emph{defining graph} $\Gamma$ is a finite simplicial graph where every edge $\{a,b\}$ has a label in $m_{ab} \in \mathbb{Z}_{\geq 2}$.
\end{definition}

\begin{definition}
    Given a defining graph $\Gamma$, the corresponding Artin group $A_\Gamma$ is defined to be the group with the following presentation: $$A_{\Gamma} := \langle V(\Gamma) \ | \ \underbrace{aba\cdots}_{m_{ab} \text{ terms}} = \underbrace{bab\cdots}_{m_{ab} \text{ terms}}, \forall \{a, b\} \in E(\Gamma) \rangle.$$
\end{definition}

Now, we will say $A$ is an abstract Artin group if $A \cong A_\Gamma$ for some defining graph $\Gamma$. Each such defining graph $\Gamma$ will be called a defining graph for $A$. Note that we may have $\Gamma_1 \neq \Gamma_2$ such that $A_{\Gamma_1} \cong A_{\Gamma_2}$.

\begin{definition}
    We will say an abstract Artin group $A$ is:
    
    \begin{enumerate}
        \item \emph{Large-type} if $\Gamma$ has all labels at least 3, where $\Gamma$ is some defining graph for $A$.

        \item \emph{XXXL-type} if $\Gamma$ has all labels at least 6, where $\Gamma$ is some defining graph for $A$.
    
        \item \emph{Triangle-free} Artin group if $\Gamma$ does not have 3 edges arranged in a triangle, where $\Gamma$ is some defining graph for $A$.
        
        \item \emph{Connected} if $\Gamma$ is connected, where $\Gamma$ is some defining graph for $A$.
    \end{enumerate}
    
\end{definition}

Take $\Gamma_1$ a defining graph with all labels at least 3, then it is known if $A_{\Gamma_1} \cong A_{\Gamma_2}$, $\Gamma_2$ also has all labels at least 3 \cite{martin2024characterising}. Moreover, if $\Gamma_1$ also has all labels at least 6 (resp. is triangle free, is connected) then the same is true of $\Gamma_2$ \cite{vaskou2023isomorphism}. As such, in each of the previous definitions one could replace the word  ``some'' with ``any''.

We now turn to a distinguished class of subgroups of Artin groups.

\begin{definition}\label{definitionParabolic}
    Given $\Gamma'$ a full subgraph of $\Gamma$, the generators corresponding to the vertices of $\Gamma'$ generate a subgroup of $A_\Gamma$ isomorphic to $A_{\Gamma'}$ \cite{van1983homotopy}, which we call a \emph{standard parabolic subgroup}.

    We call conjugates of standard parabolic subgroups \emph{parabolic subgroups}.
\end{definition}

We may notate standard parabolics by $A_S$, where $S$ is the subset of the vertices generating them.

One class of parabolic subgroups we will be interested in is the spherical parabolic subgroups. These are those $A_{\Gamma'}$ which are spherical Artin groups, in the following sense.

\begin{definition}
    An Artin group $A_\Gamma$ is called \emph{spherical} if the Coxeter group arising by adding the relation $s^2$ to the presentation for each standard generator $s$ is finite. 
\end{definition}

Notice that this definition depends on the choice of defining graph. Given the solution to the isomorphism problem for large-type Artin groups this will not be a concern for us, as being spherical is a group property within this subclass.

\begin{lemma}
    Let $A_\Gamma$ be a large-type Artin group. Then $A_\Gamma$ is spherical if and only if $\Gamma$ is a single vertex, or two vertices joined by an edge.
\end{lemma}

This motivates our interest in dihedral Artin groups. For $m \in \mathbb{N}_{\geq 3}$ we will write, $$DA_m = \langle s, t \mid  \ \underbrace{sts\cdots}_{m \text{ terms}} = \underbrace{tst\cdots}_{m \text{ terms}} \rangle,$$ for the Artin group defined by a single edge labelled by $m$. We will sometimes call spherical dihedral Artin groups odd or even based on whether $m$ is odd or even.

In the study of spherical Artin groups, a crucial tool is the Garside structure. We will not need this theory, but will borrow the following notation.

\begin{notation}
    For a spherical dihedral Artin group $DA_m$ generated by $s$ and $t$, write $\Delta_{st} := \underbrace{sts\cdots}_{m \text{ terms}}$. 
\end{notation}

\begin{lemma}\label{lemmaOutDihedral}
    If $m \geq 3$ is odd, then $\Out(DA_m) \cong C_2$. If $m \geq 4$ is even, then $\Out(DA_m) \cong C_2 \times D_\infty$.
\end{lemma}
\begin{proof}
    After applying the isomorphisms to see dihedral Artin groups as generalised Baumslag Solitar groups (these are well known, see for instance \cite[Lemma 3.4]{jones2024fixedsubgroupsartingroups}), these were calculated in \cite{ghrm2000treeaction}.
\end{proof}

\begin{lemma}\label{finiteIndexDihedralOut}
    Suppose $H \leq \Out(DA_{2m})$ fixes $\langle b \rangle$ up to conjugation. Then $H$ is finite. 
\end{lemma}
\begin{proof}
    We apply the isomorphism, $$\langle a, b \mid (ab)^n = (ba)^n \rangle \cong \langle x, t \mid x^n = tx^nt^{-1} \rangle,$$ given by $b \mapsto t$ and $a \mapsto xt^{-1}$.
    
    By \cite[Theorem D]{ghrm2000treeaction}, $\Out(DA_n)$ is generated by inner automorphisms and the following: \begin{align*}
        \alpha &: x \mapsto x^{-1}, t \mapsto t\\
        \beta &: x \mapsto x, t \mapsto t^{-1}\\
        \gamma &: x \mapsto x, t \mapsto tx,\\
    \end{align*} and in particular any outer automorphism class has a unique representative $\alpha^i\beta^j\gamma^k$ where $i, j \in \{0,1\}$ and $k \in \mathbb{Z}$. 

    Note if $\phi = \alpha^i\beta^j\gamma^k$ fixes $\langle t \rangle$ (the image of $b$) up to conjugation, then the induced map on the abelianisation (which one easily checks is $\langle \bar{x}, \bar{t} \mid [\bar{x},\bar{t}] \rangle$) fixes $\langle \bar{t} \rangle$. In particular, this means $k = 0$. The result follows. 
\end{proof}

An Artin group $A_\Gamma$ can split as an amalgamated product over standard parabolic subgroups in a way that can be read from the graph. This is known as a \emph{visual splitting}. In particular, given induced subgraphs $\Gamma_1$, $\Gamma_2$ such that $\Gamma_1 \cup \Gamma_2 = \Gamma$, $A_{\Gamma} = A_{\Gamma_1} *_{A_{\Gamma_1 \cap \Gamma_2}} A_{\Gamma_2}$. 

These splittings will be important: the space that $\Out(A)$ will later act on is essentially parameterised by choices of $\Gamma$ where $A_\Gamma \cong A$ and (iterated) visual splittings of each such $A_\Gamma$.

Finally, we will need some results about normalisers and centralisers of parabolic subgroups. All of the following results are direct consequences of previous work, but we restate them in a common way. In some cases the statement cited is more general, and we extract only the results we need.

\begin{lemma}\cite[Corollary 4.12]{godelle2007normalisers}\label{lemNormalisers}
    Suppose $A_\Gamma$ is a large-type Artin group and $A_\Lambda \leq A_\Gamma$ is a parabolic subgroup with at least 2. Suppose that $g \in A_\Gamma$ and $gA_\Lambda g^{-1} = A_{\Lambda'}$. Then $\Lambda = \Lambda'$ and $g \in A_\Lambda$.
\end{lemma}

\begin{lemma}\cite[Corollary 34]{cumplido2023centraliser}\label{lemCentralisers}
    Suppose $A_\Gamma$ is a large-type Artin group, and $s$ is a standard generator. Suppose that $g \in A_\Gamma$ and $gA_sg^{-1} \leq A_s$. Then $g \in Z(s) = \langle s \rangle \times F$, where $F$ is a finitely generated free group that can be described explicitly.
\end{lemma}
\begin{proof}
    This is exactly \cite[Corollary 34]{cumplido2023centraliser}, along with the observation that if $gsg^{-1} = s^k$, then $gsg^{-1} = s$, due to the conjugation-invariant homomorphism from $A_\Gamma$ to $\mathbb{Z}$ sending each standard generator to 1.
\end{proof}

\begin{lemma}\cite[Corollary 4.2]{paris1997parabolic}\label{lemConjugators}
    Suppose $A_\Gamma$ is a large-type Artin group, and $s,t$ are standard generators. Suppose that $g \in A_\Gamma$ and $gA_tg^{-1} \leq A_s$. Then in $\Gamma$ there exists a path of odd edges of $\Gamma$ through vertices $s = s_0, s_1 \dots s_n = t$, and $g \in Z_{A_\Gamma}(A_s)\Delta_{s_0,s_1}\Delta_{s_1,s_2}\dots\Delta_{s_{n-1}, s_n}$.
\end{lemma}
\begin{proof}
    As in the previous lemma, the conjugation invariant homomorphism to $\mathbb{Z}$ sending each generator to 1 shows that $gsg^{-1} = t$.

    It follows from \cite[Corollary 4.2]{paris1997parabolic} that if $s$ and $t$ are conjugate, then there is an odd path of edges as in the statement and $\Delta_{s_0,s_1}\Delta_{s_1,s_2}\dots\Delta_{s_{n-1}, s_n}$ is a conjugates $t$ to $s$.

    Finally, if $gtg^{-1} = s$, then $g(\Delta_{s_0,s_1}\Delta_{s_1,s_2}\dots\Delta_{s_{n-1}, s_n})^{-1}$ conjugates $s$ to itself, which exactly means it centralises $A_s$.
\end{proof}

\subsection{Deformation spaces}\label{sectionDeformation}

For this section, we follow \cite{deformatio72001guirardel}.

A deformation space is a collection of actions of a group $G$ on simplicial trees, where the trees are related by moves called elementary deformations. We will say $(T, \Omega)$ is a $G$-tree if $T$ is a simplicial tree with a $G$-action $\Omega: G \rightarrow \Aut(T)$. We will assume that our $G$-trees are minimal (there is no proper, $G$-invariant subtree). We will further assume that $T$ does not have valence two vertices. Often, for notational convenience, we will write $T$ for the $G$-tree $(T, \Omega)$, suppressing reference to the action. 

\begin{remark}\label{bettiNumber}

For simplicity of exposition, we restrict to the case that the quotient graph $T/G$ is a tree. This is a standing assumption throughout this section, and the results we state do not all hold for general deformation spaces.

The Betti number of the quotient graph is invariant for different $G$-trees in the same deformation space \cite[Section 4]{deformatio72001guirardel}, so it is enough to check that $T/G$ is a tree for one $G$-tree to justify this assumption. We will see in the next section that the deformation spaces we build satisfy this assumption.

\end{remark}

The appropriate equivalence relation for $G$-trees is \emph{$G$-equivariant isometry}. 

\begin{definition}
    If $(T_1, \Omega_1)$ and $(T_2, \Omega_2)$ are $G$-trees, an isometry $f: T_1 \rightarrow T_2$ is said to be \emph{$G$-equivariant} if for all $g \in G$, $f \circ \Omega_1(g) = \Omega_2(g) \circ f$. 

\end{definition}

Notice if $\iota_h$ is conjugation by $h$, then $(T, \Omega)$ and $(T, \Omega \circ \iota_h)$ are equivariantly isometric via the isometry $\Omega(h)$. This observation will allow an action of $\Aut(G)$ on a deformation space to descend to an action of $\Out(G)$.

When a group $G$ acts on a tree $T$, we will write $G_v$ for the stabiliser of $v \in V(T)$.

We say an edge $e = \{u, v\}$ of a $G$-tree $T$ is collapsible if $G_u \leq G_v$. The corresponding \emph{elementary collapse} corresponds to obtaining a new $G$-tree $T'$ by removing $e$ and identifying $u$ and $v$, and propagating this move equivalently across $T$. The inverse of an elementary collapse is called an \emph{elementary expansion}. A finite sequence of elementary collapses and expansions is called an \emph{elementary deformation}.

\begin{definition}\label{defDeformation}
    Given a $G$-tree $T$, the corresponding deformation space $\mathcal{D}$ is the simplicial realisation of the poset of $G$-trees (up to $G$-equivariant isometry) obtained from $T$ by elementary deformation, with $T' \leq T$ if $T$ admits a sequence of elementary collapses to $T'$.
\end{definition}

\begin{remark}
    What we call a deformation space may more accurately be called the simplicial spine. By allowing the edge lengths of the simplicial trees to vary, one obtains the ambient deformation space, which deformation retracts onto this spine. For our purposes, the spine will be sufficient and we will make use of the combinatorial structure.
\end{remark}

The following is a useful characterisation of when two trees are in the same deformation space. In the following, for a $G$-tree $T$, we will say $H \leq G$ is an elliptic subgroup if it acts on $T$ with a fixed point.

\begin{theorem}\cite[Theorem 1.1]{spaces2001forester}\label{lengthFunctionCharacterisation}
    Let $G$ be a group and $T_1$ and $T_2$ be co-compact $G$-trees. Then the following are equivalent: \begin{enumerate}
        \item $T_1$ and $T_2$ are related by an elementary deformation.
        \item $T_1$ and $T_2$ have the same elliptic subgroups.
    \end{enumerate}
\end{theorem}

We now study some distinguished trees in the deformation space.

\begin{definition}
    A $G$-tree $T$ is \emph{reduced} if no elementary collapse moves are possible.
\end{definition}

\begin{definition}
    A \emph{slide move} on a $G$-tree $T$ is a specific kind of elementary deformation consisting of the following sequence of an elementary expansion followed by an elementary collapse.

    Take $e = \{v,u\}$ an edge in $T$, with edges $\{f_i\}_{i \in I}$ distinct, and in distinct orbits, all with $v$ as an endpoint, such that for all $i \in I$, $G_{f_i} \leq G_e$. Now perform an elementary expansion at $v$ along $G_e$, to obtain an intermediate tree $T'$ with edges $e' = \{v,v'\}$, $e = \{v', u\}$, and each $f_i$ having an endpoint at $v'$. Finally collapse $e$ to complete the slide. The effect is of sliding the edges $f_i$ along the edge $e$.
\end{definition}

Slide moves are important due to the following theorem. In the following, the running assumption that the quotient of each $G$-tree by $G$ is a tree is used.

\begin{theorem}\label{slideMoves}\cite[Theorem 7.2]{deformatio72001guirardel}
     Suppose $T_1$, $T_2$ are reduced $G$-trees in the same deformation space $\mathcal{D}$. Then one can get from $T_1$ to $T_2$ by a series of slide moves, passing only through reduced $G$-trees.
\end{theorem}

One attractive feature of deformation spaces is that they are contractible. For us, this will mean that a group $G$ acting on $\mathcal{D}$ is the fundamental group of the complex of groups arising from the quotient.

\begin{theorem}\label{dContractible}\cite[Theorem 6.1]{deformatio72001guirardel}
    Any deformation space $\mathcal{D}$ is contractible.
\end{theorem}

Our final tool will be a subcomplex of $\mathcal{D}$ constructed by Guirardel and Levitt, which is a deformation retract of $\mathcal{D}$ and therefore still contractible, and is finite dimensional. Again, here the running assumption that $T/G$ is a tree is used.

\begin{definition}
    Given a $G$-tree $T$, we say (the orbit of) an edge $e$ is \emph{surviving} if there exists a sequence of elementary collapses from $T$ to a reduced tree, such that $e$ is not collapsed.

    We say $T$ is \emph{surviving} if every orbit of edges is surviving.
\end{definition}

\begin{theorem}\label{survivingSubcomplex}\cite[Theorem 7.6]{deformatio72001guirardel}
    Consider a deformation space $\mathcal{D}$ and let $\mathcal{F}$ be the subspace consisting of surviving $G$-trees. Then $\mathcal{F}$ is a finite dimensional deformation retract of $\mathcal{D}$.

    Moreover, if $\mathcal{D}$ is $\Out(G)$ invariant, so is $\mathcal{F}$.
\end{theorem}

\subsection{Type VF}\label{sectionVF}

Being type F is the strongest of the finiteness conditions. It in particular implies finite generatability and finite presentability.

\begin{definition}\label{defF}
    A group $G$ is of \emph{type F} if there is a finite Eilenberg-MacLane space; that is, a CW-complex $X$ with finitely many cells such that $\pi_1(X) \cong G$ and each higher homotopy group is trivial. 
\end{definition}

We note that being type F is closed under finite direct products (by taking the direct product of the Eilenberg-MacLane spaces) and taking finite index subgroups (by taking a finite cover).

A group with torison can never be of type F, so the strongest condition we can ask for a group with torsion is that it has a (necessarily torsion free) finite index subgroup of type F. This motivates the definition of type VF.

\begin{definition}\label{defVF}
    We say $G$ is of \emph{type VF} if it has a finite index subgroup of type F.
\end{definition}

The following theorem, applied with $X$ being a deformation space, will be our main tool for showing outer automorphism groups of certain Artin groups are of type VF. In the following \emph{rigid} means any cell fixed setwise by the action of some $g \in G$ is fixed pointwise by $g$.

\begin{theorem}\label{complexOfTypeF}\cite[Theorem 7.3.4]{geoghegan2007topological}
    Let $X$ be a contractible rigid $G$-CW complex whose quotient $X/G$ is finite. Let the stabilizer of each cell of $X$ have type F. Then G has type F.
\end{theorem}

\section{Deformation spaces for Artin groups}\label{sectionDeformationOfArtin}

In this section, given a large-type Artin group $A$, we construct a deformation space $\mathcal{D}$. Trees in this space will correspond to visual splittings of $A_\Gamma \cong A$ as amalgamated free products of \emph{chunks} of $\Gamma$. These chunks lack separating simplices in the following sense. Note that in the following definition, $\mathrm{Star}$ is the open simplicial star.

\begin{definition}
    Given a connected graph $\Gamma$, a subgraph $\Gamma'$ is \emph{separating} if $\Gamma \setminus\mathrm{Star}(\Gamma')$ is disconnected. 
\end{definition}

\begin{definition}\label{defChunk}
    Let $\Gamma$ be a defining graph. An induced subgraph $\Gamma'$ of $\Gamma$ is called a \emph{chunk} if it is connected and has no separating vertices or separating edges, and it is maximal (with respect to inclusion) with these properties.

    A \emph{chunk parabolic subgroup} of $A_\Gamma$ is a parabolic subgroup conjugate to $A_{\Gamma'}$, where $\Gamma'$ is a chunk of $\Gamma$.
\end{definition}

Figure \ref{fig:exampleChunks} shows how a defining graph can be decomposed into chunks.

\begin{definition}\label{rigidChunks}
    An Artin group $A_\Gamma$ is said to have \emph{rigid chunks} (we may also say it is chunk rigid) if the following hold.

    \begin{enumerate}
        \item Given a chunk parabolic subgroup $gA_{\Gamma'}g^{-1}$ (where $g \in A_\Gamma$ and $\Gamma'$ is a chunk of $\Gamma$), and an isomorphism $\psi: A_\Gamma \rightarrow A_\Lambda$, $\psi(gA_{\Gamma'}g^{-1}) = hA_{\Lambda'}h^{-1}$, where $h \in A_\Lambda$ and $\Lambda'$ is a chunk of $\Lambda$.
        \item For each chunk $\Gamma'$ of $\Gamma$, $\Out(A_{\Gamma'})$ is finite, or $\Gamma'$ is an even edge.
    \end{enumerate}

\end{definition}

Essentially, the first condition for being chunk rigid means that a group element being conjugate into a chunk is an algebraic property, not depending on the choice of defining graph. This is of fundamental importance to Theorem \ref{canonical}.

The second condition will help us pass to a finite index subgroup of $\Out(A)$, and find the point stabilisers in the deformation space we introduce. We expect this condition could be weakened slightly, but no weakening admits any new classes of Artin groups to be proven chunk rigid, and we expect this remains true in general.

\begin{conjecture}  
    Connected large-type Artin groups have rigid chunks.
\end{conjecture}

This conjecture is known to hold in two cases, by work of other authors.

\begin{proposition}\label{tfChunks}\cite[Theorem 1, Proposition 23]{crisp2005automorphisms}\cite[Corollary 4.18]{vaskou2023isomorphism}
    Connected large-type triangle-free Artin groups have rigid chunks.
\end{proposition}
\begin{proof}
    Following Crisp, we partition the chunks into dihedral chunks (which are just an edge) and solid chunks (which have at least two edges). Crisp shows that the solid chunks are permuted by isomorphisms \cite[Proposition 23]{crisp2005automorphisms}. Vaskou shows that dihedral parabolic subgroups are permuted \cite[Corollary 4.18]{vaskou2023isomorphism}, and since isomorphisms are injective they cannot send dihedral chunks into solid chunks, so the dihedral chunks are also permuted.

    Crisp also shows that non-dihedral chunks have finite outer automorphism groups \cite[Theorem 1]{crisp2005automorphisms}.
\end{proof}

\begin{proposition}\label{xxxlchunks}\cite[Proposition 7.1, Theorem 9.6]{Blufstein2024}
    Connected XXXL-type Artin groups (that is, those where every label is at least 6) have rigid chunks.
\end{proposition}

The conjecture has no hope of being true for all Artin groups, since it fails for RAAGs. An easy to see obstruction is that we may have chunks which are $n$-cliques, so $\Out(A_{\Gamma'}) = GL_n(\mathbb{Z})$ is not finite. It is possible that this issue could be worked around, but RAAGs also fail the first (and more important) criteria for having rigid chunks. Consider $A_\Gamma = \langle a, b, c \mid [a,b], [b,c] \rangle$. There is an automorphism given by $a \mapsto a$, $b \mapsto b$, $c \mapsto ca$, which sends the chunk $A_{bc}$ to the subgroup generated by $b$ and $ca$. This subgroup is not even parabolic so has no hope of being a chunk. The existence of such automorphisms means that the deformation space we construct has no hope of being canonical for RAAGs, and thus could at best be used to study a subgroup of the outer automorphism group.

We now define the deformation space $\mathcal{D}_\Gamma$ for each Artin group $A \cong A_\Gamma$. A priori this will depend on a choice of defining graph $\Gamma$, on the isomorphism $A \cong A_\Gamma$, and on the choices made when defining the edges in the following definition. If $A$ has rigid chunks the resultant space, which we shall simply call $\mathcal{D}$, will be independent of all of these choices (see Theorem \ref{canonical}).

\begin{definition}\label{defTGamma}
    Let $A \cong A_\Gamma$ be a connected large-type Artin group. We define $\mathcal{D}_\Gamma$, the corresponding \emph{deformation space of visual splittings}, via constructing an $A$-tree $T_\Gamma$, and taking the corresponding deformation space containing it.

    Let $\{\Gamma_i\}_{i \in I}$ be the chunks of $\Gamma$, where $I = \{1 \dots n\}$ is a finite indexing set.

    We build a graph of groups for $A_\Gamma$, and see it as a graph of groups for $A$ via the isomorphism. Let $\{v_i\}_{i \in I}$ be the vertex set, so the vertices are indexed by the chunks, and correspondingly make the chunk parabolic $A_{\Gamma_i}$ the vertex group at each vertex $v_i$.
    
    Now we turn this into a graph of groups one edge at a time. Repeat the following process while the graph of groups is not connected: choose two $v_i$, $v_j$ in separate connected components such that the chunks $\Gamma_i$ and $\Gamma_j$ share an edge $\{s,t\}$, then add an edge $\{v_i, v_j\}$ with edge group $A_{st}$ (and the induced inclusion maps); if no such pair exists, choose $v_i$, $v_j$ in separate connected components such that $\Gamma_i$, $\Gamma_j$ share a vertex $s$, and similarly add $\{v_i, v_j\}$ with edge group $s$. 

    Write $T_\Gamma$ for the universal cover of the corresponding graph of groups, seen as an $A$-tree via the isomorphism $A_\Gamma \cong A$. Take $\mathcal{D}_\Gamma$ to be the deformation space containing $T_\Gamma$.
\end{definition}

The process in Definition \ref{defTGamma} terminates since the number of connected components in the partially constructed graph of groups strictly decreases with each iteration. When this happens the resultant graph with vertex set $\{v_i\}_{i \in I}$ is connected since $\Gamma$ was connected and a union of the chunks, and no two chunks can share more than an edge or their union would be a chunk (contradicting maximality). Finally, the graph-of-groups is an (iterated) visual splitting of $A_\Gamma$, so it is indeed a graph of groups for $A_\Gamma$.

Note that since the graph of groups we built for $A$ is a tree, by Remark \ref{bettiNumber} the quotient of every $A$-tree in $\mathcal{D}_\Gamma$ is a tree, and the simplified theory of deformation spaces we considered in the previous section is sufficient for understanding this space of visual splittings.

\begin{example}\label{exampleChunking}
    Consider the defining graph $\Gamma$ as in Figure \ref{fig:exampleChunks}, for simplicity with all edges labeled 7 (so that Proposition \ref{xxxlchunks} applies). It has 4 chunks, coloured red, blue, yellow and green. The purple edge is shared by the red and blue chunks.
    
    The graph of groups for the tree $T_\Gamma$ obtained from the process in Definition \ref{defTGamma} is schematically depicted in the figure, with the vertex stabilisers indicated by colour. The purple edge has the purple dihedral parabolic as its stabiliser, and the black edges have the cyclic parabolic generated by the standard generator included in all four chunks.
    
    Notice we could equally well have had one or both of the yellow and green vertices connected to the blue vertex instead of the red one.

    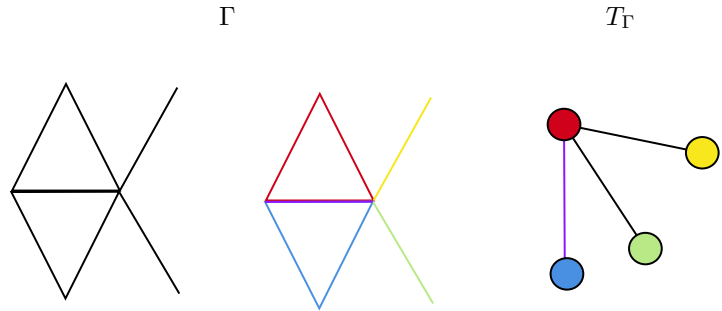
\begin{figure}[h]
    \centering

\tikzset{every picture/.style={line width=0.75pt}} 

\begin{tikzpicture}[x=0.75pt,y=0.75pt,yscale=-1,xscale=1]

\draw   (127.17,102.2) -- (154.33,156) -- (100,156) -- cycle ;
\draw   (126.92,210.47) -- (99.61,156.74) -- (153.94,156.6) -- cycle ;
\draw    (154.33,156) -- (183.33,104) ;
\draw    (153.94,156.6) -- (184.33,208) ;
\draw  [color={rgb, 255:red, 208; green, 2; blue, 27 }  ,draw opacity=1 ] (255.17,107.2) -- (282.33,161) -- (228,161) -- cycle ;
\draw  [color={rgb, 255:red, 74; green, 144; blue, 226 }  ,draw opacity=1 ] (254.92,215.47) -- (227.61,161.74) -- (281.94,161.6) -- cycle ;
\draw [color={rgb, 255:red, 248; green, 231; blue, 28 }  ,draw opacity=1 ]   (282.33,161) -- (311.33,109) ;
\draw [color={rgb, 255:red, 184; green, 233; blue, 134 }  ,draw opacity=1 ]   (281.94,161.6) -- (312.33,213) ;
\draw [color={rgb, 255:red, 144; green, 19; blue, 254 }  ,draw opacity=1 ]   (227.61,161.74) -- (281.94,161.6) ;
\draw [color={rgb, 255:red, 144; green, 19; blue, 254 }  ,draw opacity=1 ]   (378.35,122.64) -- (378.75,199.09) ;
\draw    (378.35,122.64) -- (419.39,185.28) ;
\draw    (378.35,122.64) -- (448.08,136.92) ;
\draw  [fill={rgb, 255:red, 74; green, 144; blue, 226 }  ,fill opacity=1 ] (371.57,198.06) .. controls (371.57,193.67) and (375.27,190.11) .. (379.82,190.11) .. controls (384.38,190.11) and (388.07,193.67) .. (388.07,198.06) .. controls (388.07,202.44) and (384.38,206) .. (379.82,206) .. controls (375.27,206) and (371.57,202.44) .. (371.57,198.06) -- cycle ;
\draw  [fill={rgb, 255:red, 208; green, 2; blue, 27 }  ,fill opacity=1 ] (370.1,122.64) .. controls (370.1,118.26) and (373.79,114.7) .. (378.35,114.7) .. controls (382.9,114.7) and (386.6,118.26) .. (386.6,122.64) .. controls (386.6,127.03) and (382.9,130.59) .. (378.35,130.59) .. controls (373.79,130.59) and (370.1,127.03) .. (370.1,122.64) -- cycle ;
\draw  [fill={rgb, 255:red, 248; green, 231; blue, 28 }  ,fill opacity=1 ] (439.84,136.92) .. controls (439.84,132.53) and (443.53,128.98) .. (448.08,128.98) .. controls (452.64,128.98) and (456.33,132.53) .. (456.33,136.92) .. controls (456.33,141.31) and (452.64,144.86) .. (448.08,144.86) .. controls (443.53,144.86) and (439.84,141.31) .. (439.84,136.92) -- cycle ;
\draw  [fill={rgb, 255:red, 184; green, 233; blue, 134 }  ,fill opacity=1 ] (411.14,185.28) .. controls (411.14,180.89) and (414.84,177.33) .. (419.39,177.33) .. controls (423.95,177.33) and (427.64,180.89) .. (427.64,185.28) .. controls (427.64,189.66) and (423.95,193.22) .. (419.39,193.22) .. controls (414.84,193.22) and (411.14,189.66) .. (411.14,185.28) -- cycle ;

\draw (398,61.6) node [anchor=north west][inner sep=0.75pt]   [align=left] {$\displaystyle T_{\Gamma }$};
\draw (203,61.6) node [anchor=north west][inner sep=0.75pt]   [align=left] {$\displaystyle \Gamma $};

\end{tikzpicture}

\caption{From left to right: the defining graph $\Gamma$ from Example \ref{exampleChunking}, the same graph with the chunks distinguished by colour, and the graph of groups of a resulting splitting $T_\Gamma//A_\Gamma$.}
    \label{fig:exampleChunks}
\end{figure}
\end{example}

Of course, we made choices when building $T_\Gamma$, and began with a choice of isomorphism $A \cong A_\Gamma$. It could even be that there was a choice of defining graphs $\Gamma$. This is typical of the theory of deformation spaces: when there is no canonical splitting, we try to at least find a canonical deformation space. The next theorem verifies that none of these choices changed which space we end up with, which we are hence justified in calling $\mathcal{D}$.

\begin{theorem}\label{canonical}
    If $A$ is connected, large-type Artin group with rigid chunks, with $A \cong A_\Gamma$. Then the construction of $\mathcal{D}_\Gamma$ did not depend on the choices in the construction of $T_\Gamma$, nor the choice of $\Gamma$ and isomorphism $A \cong A_\Gamma$.

    Moreover, $\Out(A)$ acts on $\mathcal{D}$.
\end{theorem}
\begin{proof}
    We repeatedly use Theorem \ref{lengthFunctionCharacterisation}.
    
    Notice that $H \leq A_\Gamma$ fixes a point on $T_\Gamma$ if and only if it fixes a vertex, which occurs if and only if it is contained in some chunk parabolic subgroup of $A_\Gamma$. Immediately we see there is no dependence on the choices made in constructing $T_\Gamma$.

    Suppose $A_\Lambda \cong A$ and fix an arbitrary isomorphism. We show that $T_\Gamma$ and $T_\Lambda$ are in the same deformation space. To see this simply note that $H$ (more precisely, its image in $A_\Gamma$) acts elliptically on $T_\Gamma$ if and only if it is contained in a chunk parabolic subgroup of $A_\Gamma$, which by Definition \ref{defChunk} occurs if and only if it is contained in a chunk parabolic subgroup of $A_\Lambda$, and this is true if and only if $H$ acts elliptically on $T_\Lambda$.
    
    Finally, we wish to show $\Out(A)$ acts on $\mathcal{D}$. We define the action as follows: given $[\varphi] \in \Out(A)$ and $(T, \Omega) \in \mathcal{D}$, we declare $[\varphi] \cdot (T, \Omega) = (T, \Omega \circ \varphi)$. 
    
     We claim that $(T, \Omega \circ \varphi) \in \mathcal{D}$. Note $H$ fixes a point on $(T, \Omega)$ if and only $H$ is contained in a chunk parabolic subgroup of $A$ (this is well defined since $A$ is chunk rigid). This occurs if and only if $\varphi(H)$ is contained in a chunk parabolic subgroup of $A$, because images of chunk parabolic subgroups under automorphisms are chunk parabolic subgroups by chunk rigidity. Finally, we note that $\varphi(H)$ is contained in a chunk parabolic subgroup of $A$ if and only if $\varphi(H)$ acts elliptically on $(T, \Omega)$, if and only if $H$ acts elliptically on $(T, \Omega\circ\varphi)$.

     Now, it is routine to check that the proposed action is a well defined action.
\end{proof}

We immediately apply Thoerem \ref{survivingSubcomplex}, and restrict our attention to the finite dimensional, $\Out(A)$ invariant subcomplex $\mathcal{F} \subseteq \mathcal{D}$ of surviving $A$-trees.

\begin{lemma}
    Given a connected large-type Artin group $A$ with rigid chunks, then for each isomorphism $A \cong A_\Gamma$, $T_\Gamma$ is reduced. In particular, $T_\Gamma \in \mathcal{F}$.
\end{lemma}
\begin{proof}

It is enough to show that in the graph of groups we built for $A_\Gamma \curvearrowright T_\Gamma$, none of the inclusions of edge groups into vertex groups were equalities. Suppose $e = \{u, v\}$ was such that $G_u = G_e$, then this induces an inclusion $G_u \leq G_v$. However this induces an inclusion of the set of standard generators generating $G_u$ into those generating $G_v$. However the standard generators $G_u$ were supposed to span a subgraph which is a chunk, and this inclusion contradicts the maximality in the definition of a chunk.
 
\end{proof}

\begin{lemma}\label{fEdgeStabilisers}
    Given a connected large-type Artin group $A$ with rigid chunks and an isomorphism $A \cong A_\Gamma$, for each $T \in \mathcal{F}$, each edge stabiliser is a parabolic subgroup of $A_\Gamma$ which is either cyclic or dihedral.
\end{lemma}
\begin{proof}
    First we note this is true for all of the reduced trees, since it is true for $T_\Gamma$, and the set of edge stabilisers is preserved by slide moves. By Theorem \ref{slideMoves}, every reduced tree in $\mathcal{F}$ is related to $T_\Gamma$ by a sequence of slide moves.

    Now, note that elementary collapses do not change the edge stabilisers of edges that are not collapsed. Since every edge in $T$ is surviving, its edge stabiliser appears in some reduced tree, so is of the required form.
\end{proof}

\section{Type VF for outer automorphism groups of Artin groups}

Our aim for this section is to show that, given $A$ an abstract connected large-type chunk rigid Artin group, $\Out(A)$ is type VF. From here, will will fix $A$ as an abstract connected large-type chunk rigid Artin group. We will write $\mathcal{F}$ for subcomplex of surviving trees in the canonical deformation space $\mathcal{D}$, as discussed in the previous section.

We will prove this in two parts. First we will show that $\mathcal{F}/\Out(A)$ is a finite complex. Secondly, we will show that $\Out(A)$ has a finite index subgroup that acts on $\mathcal{F}$ with type F simplex stabilisers.

\begin{theorem}\label{finitePresentability}
    Let $A$ be an abstract connected large-type chunk rigid Artin group. Then $\Out(A)$ is type VF.
\end{theorem}
\begin{proof}
    If $A$ is an even dihedral Artin group, then $\Out(A)$ is virtually $\mathbb{Z}$ \cite{ghrm2000treeaction}, and the result follows. So assume $A$ is not an even dihedral Artin group, then by Lemma \ref{finiteIndexInOut0} the subgroup $\Out_{0,+}(A)$ defined in Section \ref{defOut+0} is finite index in $\Out(A)$.

    We consider the action of $\Out_{0,+}(A)$ on $\mathcal{F}$, which is contractible by Theorem \ref{dContractible} and Theorem \ref{survivingSubcomplex}. This action preserves the partial order on the vertices coming from elementary collapses, so is rigid. 

    By Proposition \ref{propositionCofinite} there are finitely many $\Out(A)$ orbits of vertices in $\mathcal{F}$, and since $\Out_{0,+}(A) \leq_{f.i.} \Out(A)$ there are finitely many $\Out_{0,+}(A)$ orbits of vertices. Since the action is simplicial and $\mathcal{F}$ is finite dimensional, there are finitely many orbits of simplices, and $\mathcal{F}/\Out_{0,+}(A)$ is finite.

    Moreover, by Lemma \ref{lemmaFStabs}, every simplex stabiliser is of type F. It follows by Theorem \ref{complexOfTypeF} that $\Out_{0,+}(A)$ is of type F, and therefore $\Out(A)$ is of type VF.
\end{proof}

\begin{corollary}\label{application}
    Suppose $A_\Gamma$ is a connected large-type Artin group, and $\Gamma$ is either triangle-free, or has every label at least 6 (i.e. $A_\Gamma$ is XXXL-type). Then $\Out(A_\Gamma)$ is of type VF.
\end{corollary}
\begin{proof}
    This follows immediately from the theorem, Proposition \ref{tfChunks} and Proposition \ref{xxxlchunks}.
\end{proof}

\subsection{Co-finiteness}

In this subsection, we show that $\mathcal{F}$ has only finitely many orbits of simplices, under the action of $\Out(A)$. 

Our strategy will be as follows. Given an abstract defining graph $\Gamma$, we will introduce the notion of a $\Gamma$-tree, which is an $A_\Gamma$-tree satisfying certain conditions which essentially mean the structure of $\Gamma$ is reflected in the tree. Next, we will show that for each $\Gamma$, there are only finitely many $\Gamma$-trees. Finally, will show that every vertex in $\mathcal{F}$ can be viewed as a $\Gamma$-tree for some $\Gamma$ via an isomorphism $A \cong A_\Gamma$, and that any two vertices with the same $\Gamma$-tree structure are in the same $\Out(A)$-orbit. Together with the fact that there are only finitely many abstract defining graphs $\Gamma$ such that $A \cong A_\Gamma$, this will show there are finitely many orbits of vertices, and hence simplices.

\begin{definition}\label{defGammaTree}
    Let $\Gamma$ be an abstract defining graph. A \emph{$\Gamma$-tree} is an $A_\Gamma$-tree $T$, where as usual the action is minimal and there are no valence 2 vertices, with a strict fundamental domain $K$ of the following form:

    \begin{enumerate}
        \item $K$ is a finite subtree.
        \item For each chunk of $\Gamma$ spanned by a set of vertices $S$, there is exactly one vertex of $v \in K$ with $G_v = A_{S}$.
        \item For each other vertex and each edge in $K$, either there is a standard generator $s \in A_\Gamma$ such that the simplex has stabiliser $A_{s}$, or there is a pair of standard generators $s, t \in A_\Gamma$ generating a spherical, parabolic, dihedral Artin group $A_{s,t}$ which is the stabiliser of the simplex.
    \end{enumerate}

    We say two $\Gamma$-trees $T_1$ and $T_2$ are isomorphic if there is a graph isomorphism between the distinguished fundamental domains which preserves the edge and vertex groups.
\end{definition}

It is convenient for us to regard $\Gamma$-trees as $A_\Gamma$-trees with distinguished fundamental domains, but one could equivalently see them as graph of groups decompositions of $A_\Gamma$ (over \emph{standard} parabolic subgroups), with isomorphism being vertex and edge group preserving isomorphism of graphs of groups.

Lemma \ref{sameOrbit} justifies our definition of $\Gamma$-trees and isomorphism between them: if two trees in $\mathcal{F}$ can be given the structure of isomorphic $\Gamma$-trees (for the same $\Gamma$, so the definition of $\Gamma$-tree isomorphism makes sense), then they are in the same orbit under the action of $\Out(A)$.

We start with a technical lemma, which is essentially the observation that $\Gamma$-tree isomorphisms induce a graph-of-groups isomorphism on the $A_\Gamma$-trees, and hence there is an induced equivariant isometry. 

\begin{lemma}\label{eqIsoFromGammaIso}
    Suppose for $i \in \{1,2\}$ $(T_i, \Omega_i)$ are isomorphic $\Gamma$-trees (so $\Omega_i: A_\Gamma \rightarrow T_i$). Then they are equivariantly isometric: that is there is an isometry $\hat{f}: T_1 \rightarrow T_2$ such that for all $g \in A_\Gamma$, $\hat{f}\circ\Omega_1(g) = \Omega_2(g) \circ \hat{f}.$
\end{lemma}
\begin{proof}
    For clarity, within the proof we will suppress reference to the actions $\Omega_i$ and interpret them implicitly. Our goal is to produce an isometry $\hat{f}:T_1 \rightarrow T_2$ for all $g \in A_\Gamma$ and $x \in T_1$, $\hat{f}(gx) = g\hat{f}(x)$.

    By the definition of $\Gamma$-tree isomorphism, there is an isometry between the fundamental domains of the $\Gamma$-trees, $f: K_1 \rightarrow K_2$, which respects the edge and vertex groups. We define $\hat{f}: T_1 \rightarrow T_2$ extending this isometry by such that for every point $gx \in T_1$ where $g \in A_\Gamma$ and $x \in K_1$ (all points can be written thus, since $K_1$ is a fundamental domain), $\hat{f}(gx) := gf(x)$. We need to check that this is a well defined, equivariant, and an isometry.

    To see that $\hat{f}$ is well defined, note that if $gx = hy$ where $x,y \in K_1$, then it must be that $x = y$ since $K_1$ is a strict fundamental domain. Therefore $h^{-1}g \in G_x$. It follows that $h^{-1}g \in G_{f(x)}$, since $f$ is vertex group preserving, so $\hat{f}(gx) = gf(x) = hf(x) = \hat{f}(hx)$ as required. 

    It is clear that $\hat{f}$ sends any edge of $T_1$ isometrically to an edge of $T_2$. This is because any edge of $T_1$ can be written $ge$ where $g \in A_\Gamma$ and $e \subseteq K$, and $f$ is an isometry of $e$ onto an edge, so $\hat{f}$ is an isometry from the edge $ge$ to the edge $gf(e)$. 
    
    Therefore, to check that $\hat{f}$ is an isometric embedding, it is sufficient to check that it does not identify edges that share an endpoint. Suppose $e_1 = \{u_1, v\}$ and $e_2 = \{v, u_2\}$ are distinct edges of $T_1$. If the edges are in distinct orbits, then it is not hard to check, using that fact that $K_1$ is a strict fundamental domain, that there is $g \in A_\Gamma$ such that  $g^{-1}(e_1 \cup e_2) \subseteq K_1$, and the required result follows from the fact that $f$ is an isometric embedding. Otherwise, suppose $h^{-1}e_1 \subseteq K_1$, and so $h^{-1}u_2 = gh^{-1}u_1$ where $g \in G_{h^{-1}v} \setminus G_{h^{-1}u_1}$ by assumption that $u_1$ and $u_2$ are distinct. Note that $$G_{h^{-1}v} \setminus G_{h^{-1}u_1} = G_{f(h^{-1}v)} \setminus G_{f(h^{-1}u_1)} = G_{h^{-1}\hat{f}(v)} \setminus G_{h^{-1}\hat{f}(u_1)} = h^{-1}(G_{\hat{f}(v)} \setminus G_{\hat{f}(u_1)}),$$ where the crucial equality is the first where we use that $f$ respects the vertex groups. In particular, $$hgh^{-1} \in G_{\hat{f}(v)} \setminus G_{\hat{f}(u_1)}$$. Now, $$\hat{f}(u_2) = \hat{f}(hgh^{-1}u_1) = hgf(h^{-1}u_1) = hgh^{-1}\hat{f}(u_1),$$ and $hgh^{-1} \in G_{\hat{f}(v)}$, so $hgh^{-1}$ takes $\hat{f}(e_1)$ to $\hat{f}(e_2)$, but $hgh^{-1} \notin G_{\hat{f}(u_2)}$, so these edges are distinct as required.

    That $\hat{f}$ is surjective follows immediately from the fact any point in $T_2$ is in a translate of the fundamental domain, so can be written as $gf(x) = \hat{f}(gx)$.

    Finally, equivariance of $\hat{f}$ follows almost immediately from the definition. 
\end{proof}

\begin{lemma}\label{sameOrbit}
    For $i \in \{1,2\}$ let $(T_i, \Omega_i)$ be two $A$-trees (so $\Omega_i : A \rightarrow \Aut(T_i)$) in $\mathcal{F}$. Suppose $\Gamma$ is an abstract defining graph, and there are isomorphisms $\psi_i: A_\Gamma \rightarrow A$. Suppose that the $A_\Gamma$-trees $(T_i, \Omega_i \circ \psi_i)$ can be seen as isomorphic $\Gamma$-trees (by choice of fundamental domains $K_i$).

    Then $T_1, T_2 \in \mathcal{F}$ are in the same orbit under the action of $\Out(A)$.
\end{lemma}
\begin{proof}
    We have an automorphism $\psi_1\psi_2^{-1}: A \rightarrow A$. Consider the $A$-tree $(\psi_1\psi_2^{-1}) \cdot (T_1, \Omega_1) = (T_1, \Omega_1 \circ \psi_1\psi_2^{-1})$. Clearly it is in the orbit of $(T_1, \Omega_1)$. We will produce an equivariant isometry $f: (T_1, \Omega_1 \circ \psi_1\psi_2^{-1}) \rightarrow (T_2, \Omega_2)$. Since $\mathcal{F}$ is a space of $A$-trees up to equivariant isometry, this means that $(T_1, \Omega_1)$ and $(T_2, \Omega_2)$ are in the same $\Aut(A)$-orbit. Since the $\Aut(A)$-action factors through the $\Out(A)$-action, the proof will be complete.
    
    The definition of an isomorphism of $\Gamma$-trees, from Definition \ref{defGammaTree}, induces an equivariant isometry between the fundamental domains of the $A_\Gamma$-trees $(T_1, \Omega_1 \circ \psi_1)$ and $(T_2, \Omega_2 \circ \psi_2)$. It follows by Lemma \ref{eqIsoFromGammaIso} that there is an isometry $f: T_1 \rightarrow T_2$, equivariant in the sense that for all $g \in A_\Gamma$, $$f \circ (\Omega_1 \circ \psi_1(g)) = (\Omega_2 \circ \psi_2(g)) \circ f.$$ It follows that for all $h \in A$, $$f \circ (\Omega_1 \circ \psi_1\psi_2^{-1}(h)) = \Omega_2(h) \circ f,$$ simply by taking $g = \psi_2^{-1}(h)$ in the first equation. But this exactly means that $f$ is the required equivariant isometry.
    
\end{proof}

Now we show there are finitely many $\Gamma$-trees. Eventually every vertex in $\mathcal{F}$ will have the structure of a a $\Gamma$-tree so this, along with the previous lemma, will show there are finitely many orbits.

\begin{lemma}\label{finitelyManyGammaTrees}
    Fix $\Gamma$ an abstract defining graph. Then there only finitely many $\Gamma$-trees, up to isomorphism of $\Gamma$-trees.
\end{lemma}
\begin{proof}
    We show that there are only finitely many labelled graphs meeting the conditions on $K$ in Definition \ref{defGammaTree} and can possibly arise as a fundamental domain in a minimal $A_\Gamma$-tree without valence 2 vertices.
    
    For each chunk spanned by a set of vertices $S \subseteq V(\Gamma)$ we write $v_S$ for the unique corresponding vertex in $K$ such that $G_{v_{S}} = A_S$, and let $I$ be an index set over such sets $S$.
    
    We claim that $K = \mathrm{Hull}(\{v_{i}\}_{i\in I})$. To see this, suppose for contradiction there was a nonempty subtree of $K$ outside of this hull, attached to a vertex $v \in \mathrm{Hull}(\{v_{i}\}_{i\in I})$ by an edge $e$. Let $G_v = A_S$ (we make no assumptions beyond it being standard parabolic, which follows from the definition of a $\Gamma$-tree). Now take $w$ an arbitrary edge or vertex of this extremal subtree. By definition of a $\Gamma$-tree, $w$ has stabiliser $A_{S'}$ where $S'$ is either a standard generator, or a pair of standard generators spanning an spherical dihedral parabolic subgroup. Take $s \in S'$. This generator, viewed as a vertex in $\Gamma$, occurs in a chunk of $\Gamma$. Therefore it fixes the unique vertex $v_i$ in $K$ corresponding to this chunk, and the geodesic from $w$ to $v_i$. Given that $w$ is a vertex in an extremal subtree, this geodesic passes through $v$. It follows that $s \in A_{S}$. Since $s$ and $w$ were generic, it follows that the stabiliser of every point in the extremal subtree has stabiliser including into $A_S$. This contradicts minimality of the action, completing the proof of the claim. 

    Now we show that there is a uniform bound on the number of vertices in $K$ depending only on $\Gamma$. Since $K = \mathrm{Hull}(\{v_{i}\}_{i\in I})$, we can write $K$ as a union of $|I|^2$ geodesics, in particular $K = \bigcup_{i,j \in I}[v_i, v_j]$. Note that $|I|$ is the number of chunks, and depends only on $\Gamma$.
    
    Fix $i, j \in I$ and consider the geodesic $\mu = [v_i, v_j]$. We will devise an upper bound on the number of internal vertices of $\mu$ depending only on $\Gamma$. Observe that every internal vertex has either valence 2 or valence at least 3 in $K$ (we do not consider the valence in the ambient $A_\Lambda$-tree). Suppose $v$ internal to $\mu$ has valence at least 3, then there is an edge $e$ of $K$ at $v$ not contained in $\mu$. Since $K = \mathrm{Hull}(\{v_{i}\}_{i\in I})$, there must be $k \in I \setminus \{i, j\}$ such that $e$ is the last edge in the geodesic from $v_k$ to $\mu$. Therefore there can be at most $|I| - 2$ vertices internal to $\mu$ of valence at least 3.

    Now suppose $v$ is internal to $\mu$ with valence exactly 2 in $K$, and the call $e_1, e_2$ the adjacent edges. One case is that $v \in \{v_k\}_{k \in I \setminus \{i,j\}}$, but this happens at most $|I| -2$ times. Thus we focus on the other case, so $G_v = A_S$, where $S$ is either a single standard generator, or two standard generators spanning an spherical dihedral parabolic subgroup.
    
    In the first case where $G_v$ is cyclic, since $G_{e_1}, G_{e_2} \leq G_v$, and these subgroups are also standard parabolic, it must be that $G_{e_1} = G_v = G_{e_2}$. It follows that $v$ is valence 2 in the ambient $A_\Gamma$-tree, since $K$ is a fundamental domain, so every edge adjacent to it would be $ge_1$ or $ge_2$  with $g \in G_v$, but by the equality of stabilisers $ge_1 = e_1$ and $ge_2 = e_2$. This contradicts our assumption that trees do not have valence 2 vertices.

    In the second case where $G_v$ is spherical prabolic dihedral, it can't be that $G_{e_1} = G_v = G_{e_2}$ (or we apply the same argument verbatim as the previous paragraph), so (without loss of generality) say $G_{e_1} = A_s$ and $G_v = A_{st}$. Since the subset of $\mu$ fixed by any group element is connected, each standard generator can take the role of $t$ in this situation at most twice (at the start and end of the connected subline of $\mu$ fixed by $t$). So the number of vertices of this form is at most twice the number of standard generators: again this depends only on $\Gamma$.

    Now consider an arbitrary $\Gamma$-tree, and note the distinguished fundamental domain $K$ has bounded size, so is one of finitely many graph isomorphism types. There are finitely many choices for vertex and edge subgroups, so $K$ is one of finitely many labelled graphs. Since $\Gamma$-trees are said to be isomorphic if there is a edge and vertex group graph isomorphism between the distinguished fundamental domains, there are only finitely many isomorphism classes of $\Gamma$-trees.
\end{proof}

We will now show that for every tree in $\mathcal{F}$, there is $\Gamma$ such that the tree can be given the structure of a $\Gamma$-tree, via an isomorphism $A \cong A_\Gamma$. The proof is inductive, and we may take any $T_\Gamma$ as a base case.

\begin{lemma}\label{GammaBase}
    Consider $A \cong A_\Gamma$. Then the tree $T_\Gamma$ (see Definition \ref{defTGamma}) is a $\Gamma$-tree.
\end{lemma}
\begin{proof}
    This follows immediately from the construction.
\end{proof}

For the induction step, we will show that the property of being a $\Gamma$-tree is closed under elementary expansion and collapse. This is sufficient since every pair of trees in $\mathcal{F}$ is related by a sequence of such moves. Collapse is the easy case.

\begin{lemma}\label{GammaCollapse}
    Suppose $T$ is a $\Gamma$-tree, and $T'$ is an elementary collapse of $T$, then $T'$ is a $\Gamma$-tree (for the same $\Gamma$).
\end{lemma}
\begin{proof}
    Since collapse moves are carried out equivariantly across the tree, we can take $e = \{u,v\}$ an edge in $K$ as the edge being collapsed. Without loss of generality assume that $G_u \leq G_v$. Let $\varphi: T \rightarrow T'$ be the collapsing map. 

    We take $K' := \varphi(K)$ as the fundamental domain for $T'$, which we show has the properties required by Definition \ref{defGammaTree}. Since $K'$ is $K$ with a collapsed edge, it is still a finite tree.
    
    For the second point, we need to show that in $K'$ there is a unique vertex with stabiliser $A_S$ for each chunk spanned by $S \subseteq V(\Gamma)$. First we handle existence. Notice that $u$ and $v$ are the only vertices of $K$ which are not the unique preimage of their image, because $K$ is a strict fundamental domain. As such, they are the only vertices in $K$ whose image may have different stabiliser. Suppose $i \in \{u,v\}$ and $G_i = A_S$ is a chunk parabolic subgroup. By assumption on the structure of vertex stabilisers in $K$, and the fact $G_u \leq G_v$, it must be that $i = v$, and $G_u \neq A_{s}$. This is because $A_S$ cannot properly include into a 1 or 2 generated parabolic subgroup, and $G_u \neq G_v$ since there is only one vertex in $K$ with $A_S$ as its stabiliser. Now, however, $\varphi(v) = \varphi(u)$ and this vertex of $K'$ has stabiliser $G_v = A_S$, so the existence part of point 2 still holds in $K'$.

    For uniqueness, assume $v'$ a vertex of $K'$ has $G_{v'} = A_S$. The preimage of $v'$ under $\varphi$ contains either 1 or 2 vertices. If it has one vertex $w$, then $G_w = G_{v'}$. If it has two vertices they are $u$ and $v$, and by the definition of an elementary collapse $G_{v'} = G_v$. Either way, there is a vertex in $\varphi^{-1}(v')$ with stabiliser $A_S$, and since $\varphi^{-1}(K') = K$ the uniqueness property passes from $K$ to $K'$.

    For the final point, once again note that for each simplex of $K'$, its stabiliser is a stabiliser in $K$, so if it is not a chunk parabolic of $A_\Gamma$, then it must be the cyclic subgroup generated by a standard generator of $A_\Gamma$ or a spherical dihedral parabolic subgroup, exactly as required.
\end{proof}

For expansion, we may have to change $\Gamma$ to some new $\Gamma'$ (with $A_{\Gamma'} \cong A_\Gamma$). This case is more delicate than the previous one, with the tricky part being picking the right $\Gamma'$ and fundamental domain on $T'$ (once we have done so, checking there is a $\Gamma'$-tree structure is laborious but not difficult).

\begin{lemma}\label{GammaExpansion}
    Suppose that $T, T' \in \mathcal{F}$ are both surviving, and $\varphi: T' \rightarrow T$ is a collapse of an edge $e = \{u,v\}$ with $G_u \leq G_v$. Suppose that $T$ is a $\Gamma$-tree. Then there is $\Gamma'$ and an isomorphism $\psi: A_{\Gamma'} \rightarrow A_{\Gamma}$, such that $T'$, when viewed as an $A_{\Gamma'}$-tree (by precomposing the action $A_\Gamma \rightarrow \Aut(T')$ with $\psi$) is a $\Gamma'$-tree.
\end{lemma}
\begin{proof}
    First, we note that by Lemma \ref{fEdgeStabilisers} every edge in $T'$ has edge stabilisers which are either cyclic or spherical dihedral parabolic subgroups of $A_{\Gamma}$. In particular, this includes the edge $e$. Since $G_e \leq G_v = G_{\varphi(e)}$ which is standard parabolic (by the $\Gamma$-tree structure on $T$), it follows that $G_e$ is parabolic in $G_v$ \cite{blufstein2023parabolicsinside}. Up to replacing $e$ by a translate sharing $v$ as an endpoint, we can suppose $G_e$ is standard parabolic in $G_v$ and thus $A_\Gamma$.

    Let $K$ be the distinguished fundamental domain making $T$ a $\Gamma$-tree. Notice that we can decompose this fundamental domain as $K = \cup_{i = 1}^n K_i$ where the $K_i$ are the maximal subtrees with the vertex $\varphi(e)$ as a leaf (in particular $K_i \cap K_j = \{\varphi(e)\}$ for any two distinct $i,j$).

    As remarked above, by the $\Gamma$-structure, $G_{\varphi(e)} = A_{S_0}$ where $A_{S_0} \leq A_\Gamma$ is a parabolic subgroup. The partition of $K$ induces a partition of $V(\Gamma) =  \bigcup_{i = 0}^n S_i$, where for each $i \neq 0$ and each $s \in S_i$ the subtree of $K$ fixed by $s$ is contained in $K_i \setminus \{\varphi(e)\}$. This union covers $V(\Gamma)$ because every generator is contained in a chunk (the corresponding parabolic fixes some vertex of $K$), and partitions because $\varphi(e)$ is fixed only by standard generators in $S_0$.

    The tree structure gives rise to a decomposition of $A_\Gamma$ as an amalgamated product of factors $A_{S_0 \cup S_i}$ (for $i \in \{1 \dots n\}$) over the factor $A_{S_0}$. This is a visual splitting of the Artin group. To see this, take $i,j \in \{1 \dots n\}$ distinct, and $s \in A_{S_i}$ and $t \in A_{S_j}$ standard generators. Both $s$ and $t$ act elliptically on $T$ and have disjoint fixed sets (separated by $\varphi(e)$, which neither $s^k$ nor $t^k$ fixes for any $k \neq 0$), so generate a free group by the ping-pong lemma.\\

    We are now ready to simultaneously construct a defining graph $\Gamma'$, an isomorphism $\psi: A_{\Gamma'} \rightarrow A_\Gamma$, and a fundamental domain $K'$ for $T'$, such that when $T'$ is viewed as a $A_{\Gamma'}$-tree via the isomorphism $\psi$, the fundamental domain $K'$ makes it an $\Gamma'$-tree.

    The graph $\Gamma'$ will have a vertex  $s'$ for each $s \in \bigcup_{i=0}^n S_i$, so is naturally partitioned into sets $S_i'$ for $i \in \{0 \dots n\}$. We will add edges from vertices in $S_i'$ to vertices in $S_j'$ if and only if $i = j$ or $i = 0$ or $j = 0$, so $A_{\Gamma'}$ will split as an amalgamated product. We will define a set of maps $\psi_i: A_{S_0' \cup S_i'} \rightarrow A_\Gamma$ such that $\psi_i$ is an isomorphism onto its image $A_{S_0 \cup S_i}$. Moreover, the $\psi_i$ will agree on $A_{S_0'}$. Therefore, by the universal property of amalgamated products, these $\psi_i$ will combine to an isomorphism $\psi: A_{\Gamma'} \rightarrow A_\Gamma$. 

    Moreover, we will build $K'$ from subtrees $K_i'$, where each $K_i'$ meets each orbit in $\varphi^{-1}(K_i)$ exactly once.

    Firstly, we add edges to $\Gamma'$ between each pair of vertices in $s',t' \in S_0'$ so that $m_{s't'} = m_{st}$. For each $i$, $\psi_i : s' \mapsto s$ for each $s' \in S_0'$.

    Now fix $i \in \{1 \dots n\}$. Consider $L_i' = \overline{\varphi^{-1}(K_i \setminus \{\varphi(e)\})}$: since $\varphi$ is injective away from $e$, this is isomorphic (as a graph) to $K_i$, and has either $v$ or $hu$ as a leaf, where $h \in G_v$.  

    The easy case is that $L_i'$ has $v$ or $u$ as a leaf: in this case we can make the subgraph of $\Gamma'$ spanned by $S_0' \cup S_i'$ isomorphic to the corresponding subgraph of $\Gamma$ spanned by $S_0 \cup S_i$, and make $\psi_i$ the induced isomorphism $s' \mapsto s$. We make $K_i' = L_i' \cup e$. The unlabelled grey edges in Figure \ref{fig:expansionTwist} represent these cases.

    \begin{figure}
        \centering
        \caption{On the left is part of $T$, with solid edges in $K$. On the right is the pre-image under $\varphi$ in $T'$.}
        \tikzset{every picture/.style={line width=0.75pt}} 

\tikzset{every picture/.style={line width=0.75pt}} 

\begin{tikzpicture}[x=0.75pt,y=0.75pt,yscale=-1,xscale=1]

\draw [color={rgb, 255:red, 155; green, 155; blue, 155 }  ,draw opacity=1 ]   (1.47,88.4) -- (42.47,75.4) ;
\draw [color={rgb, 255:red, 0; green, 0; blue, 0 }  ,draw opacity=1 ]   (42.47,75.4) -- (58.47,37.4) ;
\draw [color={rgb, 255:red, 0; green, 0; blue, 0 }  ,draw opacity=1 ] [dash pattern={on 0.84pt off 2.51pt}]  (42.47,75.4) -- (82.47,92.4) ;
\draw [color={rgb, 255:red, 155; green, 155; blue, 155 }  ,draw opacity=1 ]   (42.47,75.4) -- (54.47,104.4) ;
\draw [color={rgb, 255:red, 155; green, 155; blue, 155 }  ,draw opacity=1 ]   (190.47,93.4) -- (231.47,80.4) ;
\draw    (252.47,39.4) -- (268.47,1.4) ;
\draw  [dash pattern={on 0.84pt off 2.51pt}]  (273.47,79.4) -- (313.47,96.4) ;
\draw [color={rgb, 255:red, 155; green, 155; blue, 155 }  ,draw opacity=1 ]   (273.47,79.4) -- (285.47,108.4) ;
\draw [color={rgb, 255:red, 208; green, 2; blue, 27 }  ,draw opacity=1 ]   (231.47,80.4) -- (273.47,79.4) ;
\draw [color={rgb, 255:red, 208; green, 2; blue, 27 }  ,draw opacity=1 ] [dash pattern={on 0.84pt off 2.51pt}]  (231.47,80.4) -- (252.47,39.4) ;
\draw    (181.47,59.4) -- (104.47,60.37) ;
\draw [shift={(102.47,60.4)}, rotate = 359.27] [color={rgb, 255:red, 0; green, 0; blue, 0 }  ][line width=0.75]    (10.93,-3.29) .. controls (6.95,-1.4) and (3.31,-0.3) .. (0,0) .. controls (3.31,0.3) and (6.95,1.4) .. (10.93,3.29)   ;

\draw (17,63.4) node [anchor=north west][inner sep=0.75pt]  [font=\scriptsize]  {$\varphi ( e)$};
\draw (222,70.4) node [anchor=north west][inner sep=0.75pt]  [font=\scriptsize]  {$v$};
\draw (274,69.4) node [anchor=north west][inner sep=0.75pt]  [font=\scriptsize]  {$u$};
\draw (249,81.4) node [anchor=north west][inner sep=0.75pt]  [font=\scriptsize]  {$e$};
\draw (231,48.4) node [anchor=north west][inner sep=0.75pt]  [font=\scriptsize]  {$he$};
\draw (52,52.4) node [anchor=north west][inner sep=0.75pt]  [font=\scriptsize]  {$K_{i}$};
\draw (71,75.4) node [anchor=north west][inner sep=0.75pt]  [font=\scriptsize]  {$h^{-1} K_{i}$};
\draw (238,28.4) node [anchor=north west][inner sep=0.75pt]  [font=\scriptsize]  {$hu$};
\draw (265,14.4) node [anchor=north west][inner sep=0.75pt]  [font=\scriptsize]  {$L_{i} '$};
\draw (299,78.4) node [anchor=north west][inner sep=0.75pt]  [font=\scriptsize]  {$h^{-1} L_{i} '$};
\draw (142,46.4) node [anchor=north west][inner sep=0.75pt]  [font=\scriptsize]  {$\varphi $};

\end{tikzpicture}
        
        \label{fig:expansionTwist}
    \end{figure}
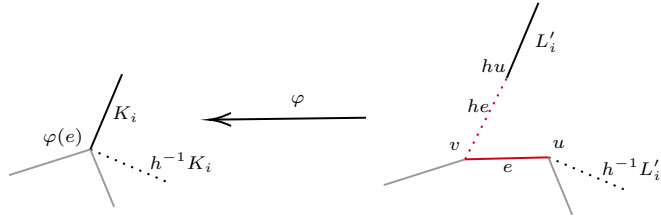

    Suppose instead that the leaf of $L_i'$ contained in the orbit of $e$ is $hu$, where $h \in G_v$. This case is represented by the black edges in Figure \ref{fig:expansionTwist}. Since $he$ is not in $K'$, it is clear if we want $K'$ to be connected we must add $h^{-1}L_i'$, instead of $L_i'$, to $K'$.
    
    Write $f$ for the edge of $L_i'$ with $hu$ as a leaf. Since $\varphi$ is equivariant and $f$ is not collapsed, $G_f = G_{\varphi(f)}$ and so is a standard cyclic or spherical dihedral parabolic, since $\varphi(f)$ is an edge of $K$. We will define $K_i' = h^{-1}L_i' \cup e$. 

    Now note that $h^{-1}f$ has $u$ as an endpoint and $h^{-1}G_fh \leq G_u$ is parabolic (in $A_\Gamma$, and therefore $G_u$). Up to replacing $h$ by $hh'$, where $h' \in G_u$, we may (and henceforth do) assume that $h^{-1}G_fh \leq G_u$ is standard parabolic. Note $h^{-1}f$ still has $u$ as an endpoint (since $h' \in G_u$).

    Now, $h \in G_v = A_{S_0}$ conjugates the standard parabolic $G_f$ of rank at most 2 to a standard parabolic. If $G_f$ is spherical dihedral, then by Lemma \ref{lemNormalisers} this means that $G_f = h^{-1}G_fh$, and $h \in G_f \leq G_u$. In particular $hu = u$, and this is the case we have already addressed.
    
    Henceforth assume $G_f$ is cyclic, generated by $s$, and hence $h^{-1}G_fh$ is cyclic generated by $t$, say. Then by Lemma \ref{lemConjugators}, $h^{-1} \in Z_{A_\Gamma}(A_t)\Delta_{s_0,s_1}\Delta_{s_1,s_2}\dots\Delta_{s_{n-1}, s_n}$, where $s_0 =t$, $s_n = s$ and for each $i$ $s_i, s_{i+1}$ span an odd parabolic subgroup. Note also that $s, t \in G_v = A_{S_0}$.

    We are ready to add edges between the vertices in $S_i'$ and those in $S_i' \cup S_0'$ to build $\Gamma'$. If $x', y' \in S_i'$, arrange that $m_{x'y'} = m_{xy}$ (so $A_{S_i'}$ is isomorphic to $A_{S_i}$). 
    
    Note that all edges from $S_i$ to $S_0$ in $\Gamma$ have $s$ as their endpoint in $S_0$. To see this take $x \in S_i$ and $r \in S_0 \setminus \{s\}$, and note that they fix points in $T$ on opposite sides of $f$, and $G_f = \langle s \rangle$ so neither fixes $f$. A ping pong argument shows that they generate a free group.
    
    In $\Gamma'$ we slide these edges, that is we arrange that $m_{x't'} = m_{xs}$ for each $x' \in S_i'$, and these are all of the edges from $S_i'$ to $S_0'$. We define $\psi_i$ by $\psi_i: x' \mapsto h^{-1}xh$ for all $x \in S_i$, and as usual $\psi_i: r' \mapsto r$ for all $r \in S_0$. By the structure of $h^{-1}$ as a product of odd dihedral Garside elements which form a path from $s$ to $t$ (when read from right to left), this is a (composition of) the well known edge twist and Dehn twist isomorphisms described for instance in \cite[Theorem 1]{crisp2005automorphisms}, so is in fact an isomorphism.\\

    We now only need to check that the action of $A_{\Gamma'}$ on $T'$ induced by the isomorphism $\psi$, with fundamental domain $K'$, makes $T'$ a $\Gamma'$-tree.

    That $K'$ is finite is clear (it has exactly one more edge than $K$). It is connected by definition (in particular, the choice of each $K_i'$ to have either $u$ or $v$ as a leaf).

    We show that $K'$ is a fundamental domain. By construction it is clear that $\varphi(K')$ meets each orbit in $T$ once. By equivariance this passes to $K'$ for the points on which $\varphi$ is injective. The only other points are those on $e$ (since $e$ is collapsed to a point in $K$), but since the action is without inversions each of these points is in a different orbit. 
    
    It remains to show that $K'$ meets every orbit. To see this note that $x \in T'$ is any point, then there are two cases. The first is that $\varphi(x)$ is a translate of $\varphi(e)$, which means $x$ is on a translate of $e$ so its orbit intersects $K'$. In the other case, $\varphi(x) = gk$ for some $k \in K \setminus \{\varphi(e)\}$, but since $\varphi$ is injective away from the orbit of $e$, the preimage of $k$ is some $k'$, which is either in $K'$, or its $h^{-1}$ translate is (if we were in the case that $L_i$ had $hu$ as an endpoint). Either way the orbit of $k'$ intersects $K'$, and by equivariance $\varphi(x) = g\varphi(k')$ implies  $\varphi(x) = \varphi(gk')$, and injectivity on these points gives $x = gk'$. 

    Let $\Lambda'$ be a chunk of $\Gamma'$. We note that, just as $T$ induced a visual splitting of $A_\Gamma$, $T'$ induces a visual splitting of $A_{\Gamma'}$. As such there must exist $i \in \{0 \dots n\}$ such that $V(\Lambda') \subseteq S_i'$, as if this were not possible the visual splitting of $A_{\Gamma'}$ would induce a visual splitting of $A_{\Lambda'}$ over a separating vertex or separating edge, since all edge stabilisers in $T'$ are of this form.
    
    By the definitions of $S_i'$ and $\psi$, $\psi(A_{\Lambda'}) = h^{-1}A_\Lambda h$, where $\Lambda$ is a chunk of $\Gamma$. Since $T$ is a $\Gamma$-tree, $A_\Lambda$ fixes a unique point in $K_i$, so $\psi(A_{\Lambda'})$ fixes a unique point in $h^{-1}K_i$. By definition of $K_i'$, and equivariance of the collapse, $\varphi$ sends $K_i'$ to $h^{-1}K_i$, so $\psi(A_{\Lambda'})$ fixes a point in $K'$.
    
    It follows from the definition of the elementary collapse that $A_{\Lambda'}$ fixes a unique point in $K'$ unless the point fixed by $A_{\Lambda'}$ does not have unique preimage under $\varphi$.
    
    The way the remaining case could go wrong is if $A_{\Lambda'} = G_v$ and $G_u = G_v$. In this case $A_{\Lambda'}$ is an edge stabiliser. Now, since $\{u, v\}$ is surviving and thus appears in a reduced tree, which differs from $T_{\Gamma'}$ by slide moves by Theorem \ref{slideMoves}, it follows that $A_{\Lambda'}$ is an edge stabiliser in $T_{\Gamma'}$. This means $\Lambda'$ includes into a distinct chunk, which contradicts that $\Lambda'$ is a chunk (in particular failing maximality). 

    The final part of Definition \ref{defGammaTree} we need to check is that every other edge and vertex stabiliser is either an spherical dihedral parabolic subgroup or cyclic parabolic subgroup. For edges this follows from Lemma \ref{fEdgeStabilisers}. For vertices, we only need to check the orbit of $u$ (by our construction every other vertex inherits this property from the unique vertex in the orbit of its image under $\varphi$ in $K$), but $G_u = G_e$ so this follows from the property for edges.
\end{proof}

We now have all of the pieces to make every tree in $\mathcal{F}$ a $\Gamma$-tree.

\begin{lemma}\label{GammaStructure}
   For each $A$-tree $(T, \Omega) \in \mathcal{F}$, there is an abstract defining graph $\Gamma$ and an isomorphism $\psi: A_\Gamma \rightarrow A$ such that $(T, \Omega \circ \psi)$ is a $\Gamma$-tree.
\end{lemma}
\begin{proof}
    This follows immediately from what we have already done by induction. For the base case, Lemma \ref{GammaBase} gives us a vertex in $\mathcal{F}$ which can be given the structure of a $\Gamma$-tree. We conclude inductively by Lemma \ref{GammaCollapse} and Lemma \ref{GammaExpansion}, since every vertex in $\mathcal{F}$ is connected to every other vertex by a finite sequence of elementary collapses and expansions.
\end{proof}

The main result for this section follows by combining the previous lemmas.

\begin{proposition}\label{propositionCofinite}
    There are finitely many $\Out(A)$-orbits of vertices in $\mathcal{F}$.
\end{proposition}
\begin{proof}
    By Lemma \ref{GammaStructure}, for every vertex $(T, \Omega) \in \mathcal{F}$, fix an abstract defining graph $\Gamma$ and an isomorphism $\psi: A_\Gamma \rightarrow A$ such that $(T, \Omega \circ \psi)$ is a $\Gamma$-tree. In some cases we have a choice, which we resolve arbitrarily.

    There is an equivalence relation on the vertices of $\mathcal{F}$, where two $A$-trees are equivalent if we chose the same defining graph, and the chosen $\Gamma$-tree structures were isomorphic. There are only finitely many equivalence classes, because there are only finitely many abstract defining graphs $\Gamma$ such that $A_\Gamma \cong A$ \cite[Theorem A]{vaskou2023isomorphism}, and for each fixed $\Gamma$ there are only finitely many isomorphism classes of $\Gamma$-tree by Lemma \ref{finitelyManyGammaTrees}. However, it follows from Lemma \ref{sameOrbit} that this equivalence relation is a refinement of the $\Out(A)$-orbit equivalence relation. So there are also only finitely many $\Out(A)$-orbits, as required.
\end{proof}

\subsection{A finite index subgroup}\label{defOut+0}

We now show how to pass to a finite index subgroup of $\Out(A)$, which we will go on to prove is of type F.

In what follows we say $H \leq A$ is a chunk parabolic if this is true of the image of $H$ under some (equivalently any) isomorphism $A \cong A_\Gamma$. This notion is well defined, by chunk rigidity. Note that the notion of a \emph{standard} chunk parabolic depends on the isomorphism, so is not well defined for $A$.

Write $\mathcal{C}$ for the set of conjugacy classes of chunk parabolic subgroups. This finite set $\mathcal{C}$ is permuted by $\Aut(A)$ and this action  descends to $\Out(A)$. Write $\Out_0(A)$ for the kernel of this action. Clearly $\Out_0(A)$ is finite index.

\begin{lemma}
    Let $H \leq A$ be chunk parabolic subgroup. Then there is a well defined homomorphism $\rho_H: \Out_0(A) \rightarrow \Out(H)$.
\end{lemma}
\begin{proof}
    Take $\varphi$ an arbitrary representative of $[\varphi] \in \Out_0(A)$. Then there is $g \in A$ such that $\varphi(H) = gHg^{-1}$. We defined $\rho_H(\varphi) = [(\iota_g^{-1} \circ \varphi)|_H]$.

    First we show this does not depend on the choice of $g$. Suppose $g_1Hg_1^{-1} = g_2Hg_2^{-1}$, then $g_2^{-1}g_1$ normalises $H$. By Lemma \ref{lemNormalisers} $g_2^{-1}g_1 \in H$, so $(\iota_{g_1}^{-1} \circ \varphi)|_H$ and $(\iota_{g_2}^{-1} \circ \varphi)|_H$ differ by an inner automorphism of $H$.
    
    Now we show the definition did not depend on our choice of $\varphi$. If $\iota_a\varphi \in [\varphi]$ (where $a \in A$) then $\iota_a\varphi(H) = agH(ag)^{-1}$. Clearly $\iota_g^{-1} \varphi$ and $\iota_{ag}^{-1}\iota_a\varphi$ are the same automorphism, so in particular restrict to the same automorphism of $H$.

    Finally to check that $\rho$ is a homomorphism, observe that if $\varphi(H) = g_1Hg_1^{-1}$ and $\psi(H) = g_2Hg_2^{-1}$, then one easily checks that $\psi\varphi(H) = (\psi(g_1)g_2)H(\psi(g_1)g_2)^{-1}$, and $\iota_{g_2}^{-1}\psi \circ \iota_{g_1}^{-1}\varphi = \iota_{\psi(g_1)g_2}^{-1}\varphi\psi$.
\end{proof}

Given $H$ and a conjugate $gHg^{-1}$, the isomorphism $h \mapsto ghg^{-1}$ induces an isomorphism on the outer automorphism groups. This isomorphism does not depend on the choice of $g$, for suppose that $gHg^{-1} = g'Hg'^{-1}$. By Lemma \ref{lemNormalisers} $g' = gh$ where $h \in H$. Then take $\varphi: H \rightarrow H$, and note that the images of $[\varphi]$ under the associated isomorphisms are $[\iota_{g}\varphi\iota_{g}^{-1}]$ and $[\iota_{gh}\varphi\iota_{gh}^{-1}]$. Now it suffices to observe that \begin{align*}
    (\iota_{g}\varphi\iota_{g}^{-1})^{-1} \circ \iota_{gh}\varphi\iota_{gh}^{-1} &= \iota_g \varphi^{-1} \iota_{g^{-1}gh} \varphi \iota_{gh}^{-1}\\
    &= \iota_{g\varphi^{-1}(h)h^{-1}g^{-1}},
\end{align*} where $\varphi^{-1}(h)$ is well defined as $h \in H$. This is an inner automorphism of $gHg^{-1}$.

We now show that $\rho_H$ actually depends only on the conjugacy class of $H$, where $\Out(H)$ and $\Out(gHg^{-1})$ are identified via the induced isomorphism just discussed.

\begin{lemma}
    Given $H$ and a conjugate $aHa^{-1}$, and $[\varphi] \in \Out_0(A)$, the isomorphism $\alpha: \Out(H) \rightarrow \Out(aHa^{-1})$ induced by $\iota_a: H \rightarrow aHa^{-1}$ identifies $\rho_H([\varphi])$ and $\rho_{aHa^{-1}}([\varphi])$.
\end{lemma}
\begin{proof}
    Suppose that $\varphi(H) = g_1Hg_1^{-1}$ and $\varphi(aHa^{-1}) = g_2Hg_2^{-1}$. Then we directly calculate, \begin{align*}
        \alpha(\rho_H([\varphi]))^{-1}\rho_{aHa^{-1}}([\varphi]) &= [(\iota_a \iota_{g_1}^{-1}\varphi \iota_a^{-1})^{-1}|_{aHa^{-1}}(\iota_{g_2}^{-1}\varphi)|_{aHa^{-1}}]\\
        &= [(\iota_{a\varphi^{-1}(g_1a^{-1}g_2^{-1})})|_{aHa^{-1}}].
    \end{align*} Now we notice that, since $a\varphi^{-1}(g_1a^{-1}g_2^{-1})$ conjugates $aHa^{-1}$ to itself, it follows by Lemma \ref{lemNormalisers} that $a\varphi^{-1}(g_1a^{-1}g_2^{-1}) \in aHa^{-1}$, and so $\alpha(\rho_H([\varphi])) = \rho_{aHa^{-1}}([\varphi])$ in $\Out(aHa^{-1})$ as required. 
\end{proof}

Recall we write $\mathcal{C}$ for the finite set of conjugacy classes of chunk parabolic subgroups. We are now define $$\rho: \Out_0(A) \rightarrow \Pi_{[H] \in \mathcal{C}} \Out(H),$$ where $\rho$ is the product of the maps $\rho_H$, and by the lemmas above this is a well defined homomorphism independent of the choice of $H$. 

\begin{lemma}\label{finiteIndexInOut0}
    Suppose $A$ is not an even dihedral Artin group. Then the image of $\rho$ in $\Pi_{[H] \in \mathcal{C}} \Out(H)$ is finite.
\end{lemma}
\begin{proof}
    It is enough to show that the image in each factor is finite. By Definition \ref{rigidChunks}, $\Out(H)$ is finite unless $H$ is a chunk parabolic subgroup coming from a single edge with an even label. Suppose $H$ is such a chunk parabolic subgroup.

    Let $\Gamma$ be a defining graph and $\psi: A \rightarrow A_\Gamma$ an isomorphism. Without loss of generality, assume $\psi(H) = A_\Lambda$, where $\Lambda$ is by assumption an even edge and a chunk in $\Gamma$. Let $s$ be a standard generator associated to a vertex of $\Lambda$ which is of valence at least 2 (this must exist since $A$ is connected and not just an edge by assumption), and write $\Lambda'$ for another chunk containing $s$. Let $t$ be the standard generator associated to the other vertex in $\Lambda$.

    By \cite[Theorem F]{vaskou2023isomorphism}, $A_s$ is sent by any automorphism to $A_r$ up to conjugation, where $A_r$ is a standard generator. For all automorphisms in $\Out_0(A)$, however, it must be that $A_r$ is conjugated into both $A_\Lambda$ and $A_{\Lambda'}$ (since these are chunk parabolics so by definition of $\Out_0(A)$ fixed up to conjugation). So in fact, by the main theorem in \cite{blufstein2023parabolicsinside}, $A_r$ is conjugated to either $A_s$ or $A_t$.
    
    Note there is no path in $\Gamma$ from $t$ to $\Lambda'$ except through $s$, as if there were there would be a cycle involving the edge $\Lambda$, contradicting the fact that $\Lambda$ is a chunk. So there is no path from $t$ to a vertex in $\Lambda'$ through only odd labeled edges (because $\Lambda$ is by assumption even labeled), so by Lemma \ref{lemConjugators} $A_t$ is not conjugate to the cyclic subgroup generated by a standard generator in $A_{\Lambda'}$. So  $A_r$ must be conjugate to $A_s$.

    It follows that in $\rho(\Out_0(A)) \subseteq \Out(H)$ is contained in the subgroup fixing one of the standard generators up to conjugation (more presicely the preimage in $\Out(H)$ of this subgroup of $\Out(A_\Lambda)$, under the isomorphism induced by $\psi|_{H}$). By Lemma \ref{finiteIndexDihedralOut}, this is a finite subgroup.
\end{proof}

Hence $\ker(\rho)$ is finite index in $\Out_0(A)$ which in turn is finite index in $\Out(A)$. We will write $\Out_{0,+}(A)$ for $\ker(\rho)$.

\subsection{Type F stabilisers}

In this subsection, we show that the subgroup $\Out_{0,+}(A)$ from the previous subsection acts on $\mathcal{F}$ with type F simplex stabilisers.

We will first restrict our attention to vertices. Given an $A$-tree $(T, \Omega)$, which we shall as usual often write $T$, we denote the stabiliser of the corresponding point in $\mathcal{F}$ by $\Out^T(A)$. If $[\varphi] \in \Out^T(A)$, it means by definition there is an equivariant isometry $f: (T, \Omega) \rightarrow (T, \Omega \circ \varphi)$. We write $\Out^T_0(A) \leq \Out^T(A)$ for the subgroup of $\Out^T(A)$ where this isometry is moreover $A$-orbit-preserving (note that the $A$-orbit equalivalence relation for $\Omega$ and $\Omega \circ \varphi$ is the same).

We write $(V, E)$ for the graph underlying the graph of groups $T//A$, and $G_s$ for the associated simplex groups where $s \in V \cup E$. Following \cite{bass96stabiliers, levitt2005pointstabilisers} we consider the homomorphism, $$\rho_T: \Out^T_0(A) \rightarrow \Pi_{v \in V} \Out(G_v),$$ where on each factor of the co-domain, $\rho_T([\varphi])$ corresponds to the automorphism $[\varphi]$ induces on the corresponding vertex groups. This definition makes sense because, by our restriction to the subgroup $\Out^T_0(A)$, there is an element of $[\varphi]$ fixing each vertex of $T$ (and thus inducing an automorphism on $G_v$), and this element is well defined up to an element of $\Inn(G_v)$, so is a well defined class in $\Out(G_v)$. 

We write $\Out_{0,+}^T(A)$ for $\ker(\rho_T)$. The following lemma shows that, as the name suggests, this is the point stabiliser of $T \in \mathcal{F}$ under the action of the group $\Out_{0,+}^T(A)$ from the previous section.

\begin{lemma}
    Given $T$ an $A$-tree in $\mathcal{F}$, $\Out_{0,+}(A) \cap \Out^T(A) = \Out_{0,+}^T(A)$.
\end{lemma}
\begin{proof}
    We first notice that $\Out^T_0(A) = \Out^T(A) \cap \Out_0(A)$, since for $[\varphi] \in \Out^T(A)$, $[\varphi] \in \Out^T_0(A)$ if and only if and only if it preserves the orbits, but this occurs if and only if $[\varphi]$ preserves the vertex orbits (as $T/A$ is a tree, and in particular simplicial), and this is equivalent to the conjugacy classes of the chunk parabolic subgroups being preserved.

    Now it is clear that $\rho|_{\Out^T_0} = \rho_T$. The result follows. 
\end{proof}

It remains to study the kernel of $\rho_T$. In general this kernel is generated by ``bitwists'', as defined in \cite{levitt2005pointstabilisers}.

\begin{definition}
    Let $T$ be a $G$-tree such that the quotient is a tree, and $e = \{u,v\}$ be an edge in $T/G$. View $G$ as an amalgam over $G_e$, with one factor containing $G_u$ and the other containing $G_v$. Then given $z \in G_u$ and $z' \in G_v$ each normalising the included copy of $G_e$ and furthermore inducing the same automorphism on $G_e$ via conjugation, the corresponding \emph{bitwist} $D_{z,z'}$ is the automorphism which acts as conjugation by $z$ on the $G_u$ factor and conjugation by $z'$ on the $G_v$ factor.

    A bitwist is called a \emph{twist} if it is of the form $D_{z,1}$.
\end{definition}

The fact that on each factor the bitwists act like inner automorphisms means $[D_{z,z'}] \in \Out_{0,+}^T(G) \leq \Out^T(G)$ for each bitwist, where $T$ is the tree with respect to which the bitwist was defined.

In our case, the group generated by the twists (or more accurately their images in $\Out(A)$), which we shall call $\mathcal{T} \leq \Out^T(A)$, is the entirety of $\Out^T_{0,+}(A)$. This group is of type F due to Levitt's explicit description of $\mathcal{T}$.

\begin{lemma}\label{lemmaFKernel}
    Let $\Out^T(A)$ be a point stabiliser in $\mathcal{F}$. Then $\Out^T_{0,+}(A)$ is of type F.
\end{lemma}
\begin{proof}
    By \cite[Proposition 2.2]{levitt2005pointstabilisers}, the bitwists generate $\Out^T_{0,+}(A)$. We will show that for each $[D_{z,z'}] \in \Out^T(A)$, $[D_{z,z'}] \in \mathcal{T}$, and thus $\Out^T_{0,+}(A) = \mathcal{T}$.

    Let $[D_{z,z'}] \in \Out^T(A)$ be the image of a bitwist of an edge $e= \{u,v\}$. Suppose $z$ and $z'$ induce an inner automorphism on $G_e$. Then we may assume $D_{z,z'}$ acts trivially on $G_e$ (as we may precompose an arbitrary representative by an inner automorphism conjugating by an element of $G_e$). But then $D_{z,z'} = D_{z,1}D_{1,z'}$ is a product of twists, so $[D_{z,z'}] \in \mathcal{T}$.
    
    We claim that for every $[D_{z,z'}] \in \Out^T(A)$, the automorphism induced by $z$ (and thus $z'$, as they induce the same automorphism) is inner. By Lemma \ref{fEdgeStabilisers} there are two cases for $G_e$. The first case is that $G_e$ is spherical parabolic dihedral (here Lemma \ref{fEdgeStabilisers} says this is true regardless of which isomorphism $A \cong A_\Gamma$ we pick, justifying our lack of choice of Artin presentation), but in this case $G_e \leq A$ is self-normalising by Lemma \ref{lemNormalisers}, so $z \in G_e$ by assumption that it normalises $G_e$, and of course it induces an inner automorphism. The second case is that $G_e = A_s$ where $s$ is a standard generator. However $s$ is not conjugate in $A$ to its inverse by Lemma \ref{lemCentralisers}, so in fact $z$ acts trivially. In either case, by the preceeding paragraph $[D_{z,z'}] \in \mathcal{T}$. This completes the proof that $\Out^T_{0,+}(A) = \mathcal{T}$.

    A presentation for $\mathcal{T}$ is given by \cite[Proposition 3.1]{levitt2005pointstabilisers}. In particular, $$\mathcal{T} = \Pi_{e \in E} Z_{A_{o(e)}}(A_e) / (\Pi_{v \in V} Z(A_v) \times \Pi_{e \in E^+} Z(A_e)),$$ where $E$ is the set of oriented edges of $T//A$ (and $E^+$ is a set of arbitrary representatives for each edge) and $o(e)$ is the origin vertex of such an edge. The centers of the edge and vertex stabilisers embed into the product of centralisers in the obvious diagonal way.

    By Lemma \ref{fEdgeStabilisers}, $A_e$ is always a spherical dihedral or cyclic parabolic subgroup of $A_{o(e)}$. If $A_e$ is dihedral then it follows from Lemma \ref{lemNormalisers} that $Z_{A_{o(e)}}(A_e) \cong \mathbb{Z}$ is the centre of $A_e$. If $A_e$ is cyclic then, by Lemma \ref{lemCentralisers}, $Z_{A_{o(e)}}(A_e) \cong \mathbb{Z} \times F$, where $F$ is a finitely generated free group, and the central $\mathbb{Z}$ is exactly $Z(A_e) = A_e$. In this case, if $A_{o(e)}$ is dihedral, then the $F$ direct factor is in fact cyclic, and generated by $Z(A_{o(e)})$.

    So $\Pi_{e \in E} Z_{A_{o(e)}}(A_e)$ is a direct product of $\mathbb{Z}$ and $\mathbb{Z} \times F$ factors. One easily checks from the preceding discussion that the quotient identifies some of the $\mathbb{Z}$ factors and free direct factors $F$ (when $F$ is cyclic). We conclude that $\mathcal{T}$ is a right angled Artin group and, in particular, is of type F.
\end{proof}

\begin{lemma}\label{lemmaFStabs}
    Let $S \subseteq \mathcal{F}$ be a simplex. Then the stabiliser of $S$ under the action of $\Out_{0,+}(A)$ is of type F.
\end{lemma}
\begin{proof}
    Given that simplices in $\mathcal{F}$ correspond to chains of collapses, take $v \in S$ to be the top vertex in the partial order. Of course, $G_S \leq G_v$. Now, consider that there are only finitely many sequences of collapses of the $A$-tree corresponding to $v$, because collapses are equivariant and there are only finitely many orbits of edges. In particular, $G_vS \subseteq \mathcal{F}$ is a finite subcomplex. It follows that $G_S \leq_{f.i.} G_v$, and $G_S \cap \Out_{0,+}(A) \leq_{f.i.} G_v \cap \Out_{0,+}(A)$. 

    The result follows since $G_v \cap \Out_{0,+}(A)$ is of type F by Lemma \ref{lemmaFKernel}, and being type F is closed under taking finite index subgroups.
\end{proof}

\bibliographystyle{amsalpha}
\bibliography{bibliography}

\providecommand{\bysame}{\leavevmode\hbox to3em{\hrulefill}\thinspace}
\providecommand{\MR}{\relax\ifhmode\unskip\space\fi MR }
\providecommand{\MRhref}[2]{%
  \href{http://www.ams.org/mathscinet-getitem?mr=#1}{#2}
}
\providecommand{\href}[2]{#2}
\begin{thebibliography}{GHMR00}

\bibitem[AC22]{an2022automorphismgroupsartingroups}
Byung~Hee An and Youngjin Cho, \emph{The automorphism groups of artin groups of edge-separated clttf graphs}, 2022.

\bibitem[BCV24]{bregman2024raagnielsen}
Corey Bregman, Ruth Charney, and Karen Vogtmann, \emph{Finite groups of untwisted outer automorphisms of raags}, 2024.

\bibitem[BJ96]{bass96stabiliers}
Hyman Bass and Renfang Jiang, \emph{Automorphism groups of tree actions and of graphs of groups}, J. Pure Appl. Algebra \textbf{112} (1996), no.~2, 109--155. \MR{1402782}

\bibitem[BMV24]{Blufstein2024}
Martín Blufstein, Alexandre Martin, and Nicolas Vaskou, \emph{Homomorphisms between large-type artin groups}, 2024.

\bibitem[BP23]{blufstein2023parabolicsinside}
Martín Blufstein and Luis Paris, \emph{Parabolic subgroups inside parabolic subgroups of artin groups}, Proceedings of the American Mathematical Society \textbf{151} (2023), no.~4, 1519--1526.

\bibitem[CC05]{charney2005finiteaffine}
Ruth Charney and John Crisp, \emph{Automorphism groups of some affine and finite type {A}rtin groups}, Math. Res. Lett. \textbf{12} (2005), no.~2-3, 321--333. \MR{2150887}

\bibitem[CL83]{Collins1983AutomorphismsAH}
Donald~J. Collins and Frank Levin, \emph{Automorphisms and hopficity of certain baumslag-solitar groups}, Archiv der Mathematik \textbf{40} (1983), 385--400.

\bibitem[CMV23]{cumplido2023centraliser}
María Cumplido, Alexandre Martin, and Nicolas Vaskou, \emph{Parabolic subgroups of large-type artin groups}, Mathematical Proceedings of the Cambridge Philosophical Society \textbf{174} (2023), no.~2, 393–414.

\bibitem[CP24]{castel2024endomorphismsartingroupstype}
Fabrice Castel and Luis Paris, \emph{Endomorphisms of artin groups of type d}, 2024.

\bibitem[Cri05]{crisp2005automorphisms}
John Crisp, \emph{Automorphisms and abstract commensurators of 2--dimensional {A}rtin groups}, Geometry \& Topology \textbf{9} (2005), no.~3, 1381--1441.

\bibitem[CSV17]{charney17untwisted}
Ruth Charney, Nathaniel Stambaugh, and Karen Vogtmann, \emph{Outer space for untwisted automorphisms of right-angled {A}rtin groups}, Geom. Topol. \textbf{21} (2017), no.~2, 1131--1178. \MR{3626599}

\bibitem[DG81]{dyer1981braid}
Joan~L. Dyer and Edna~K. Grossman, \emph{The automorphism groups of the braid groups}, American Journal of Mathematics \textbf{103} (1981), no.~6, 1151--1169.

\bibitem[DW19]{day19vfraag}
Matthew~B. Day and Richard~D. Wade, \emph{Relative automorphism groups of right-angled {A}rtin groups}, J. Topol. \textbf{12} (2019), no.~3, 759--798. \MR{4072157}

\bibitem[For01]{spaces2001forester}
Max Forester, \emph{Deformation and rigidity of simplicial group actions on trees}, Geometry \& Topology \textbf{6} (2001), 219--267.

\bibitem[Geo07]{geoghegan2007topological}
Ross Geoghegan, \emph{Topological methods in group theory}, Graduate Texts in Mathematics, Springer New York, 2007.

\bibitem[GHMR00]{ghrm2000treeaction}
N.~D. Gilbert, J.~Howie, V.~Metaftsis, and E.~Raptis, \emph{Tree actions of automorphism groups}, Journal of Group Theory \textbf{3} (2000), no.~2, 213--223.

\bibitem[GL07]{deformatio72001guirardel}
Vincent Guirardel and Gilbert Levitt, \emph{Deformation spaces of trees}, Groups, Geometry and Dynamics \textbf{1} (2007), no.~2, 135--181.

\bibitem[GL16]{guirardel2016jsj}
Vincent Guirardel and Gilbert Levitt, \emph{Jsj decompositions of groups}, Asterisque \textbf{2017} (2016).

\bibitem[God07]{godelle2007normalisers}
Eddy Godelle, \emph{Artin–tits groups with cat(0) deligne complex}, Journal of Pure and Applied Algebra \textbf{208} (2007), no.~1, 39--52.

\bibitem[HOV24]{huang2024twistless}
Jingyin Huang, Damian Osajda, and Nicolas Vaskou, \emph{Rigidity and classification results for large-type artin groups}, 2024.

\bibitem[JV24]{jones2024fixedsubgroupsartingroups}
Oli Jones and Nicolas Vaskou, \emph{Fixed subgroups in artin groups}, 2024.

\bibitem[Ker83]{kerckhoff1983nielsen}
Steven~P. Kerckhoff, \emph{The nielsen realization problem}, Annals of Mathematics \textbf{117} (1983), no.~2, 235--265.

\bibitem[Lev05]{levitt2005pointstabilisers}
Gilbert Levitt, \emph{Automorphisms of hyperbolic groups and graphs of groups}, {Geometriae Dedicata} (2005), 49--70.

\bibitem[Lev07]{levitt2007automorphism}
\bysame, \emph{On the automorphism group of generalized baumslag--solitar groups}, Geometry \& Topology \textbf{11} (2007), no.~1, 473--515.

\bibitem[MV24]{martin2024characterising}
Alexandre Martin and Nicolas Vaskou, \emph{Characterising large-type artin groups}, Bulletin of the London Mathematical Society (2024) (English).

\bibitem[Nie24]{Nielsen1924}
J.~Nielsen, \emph{Die isomorphismengruppe der freien gruppen}, Mathematische Annalen \textbf{91} (1924), 169--209.

\bibitem[Par97]{paris1997parabolic}
Luis Paris, \emph{Parabolic subgroups of artin groups}, Journal of Algebra \textbf{196} (1997), no.~2, 369--399.

\bibitem[PS24a]{paris2024endomorphismsB}
Luis Paris and Ignat Soroko, \emph{Endomorphisms of artin groups of type ${B}_n$}, 2024.

\bibitem[PS24b]{paris2024endomorphismsatilde}
\bysame, \emph{Endomorphisms of artin groups of type $\tilde {A}_n$}, 2024.

\bibitem[Vas23a]{vaskou2023automorphisms}
Nicolas Vaskou, \emph{Automorphisms of large-type free-of-infinity {A}rtin groups}, 2023.

\bibitem[Vas23b]{vaskou2023isomorphism}
Nicolas Vaskou, \emph{The isomorphism problem for large-type artin groups}, 2023.

\bibitem[VC86]{Vogtmann1986}
K.~Vogtmann and M.~Culler, \emph{Moduli of graphs and automorphisms of free groups.}, Inventiones mathematicae \textbf{84} (1986), 91--120.

\bibitem[VdL83]{van1983homotopy}
Harm Van~der Lek, \emph{The homotopy type of complex hyperplane complements}, Ph.D. thesis, Katholieke Universiteit te Nijmegen, 1983.

\bibitem[Zim81]{Zimmermann1981freenielsen}
Bruno Zimmermann, \emph{Über homöomorphismen n-dimensionaler henkelkörper und endliche erweiterungen von schottky-gruppen.}, Commentarii mathematici Helvetici \textbf{56} (1981), 474--486.

\end{thebibliography}

\end{document}